\documentclass[12pt,reqno]{amsart}

\usepackage{color}
\definecolor{red}{rgb}{1,0,0}
\definecolor{green}{rgb}{0,1,0}
\definecolor{blue}{rgb}{0,0,1}

\usepackage{amssymb}
\usepackage{amsmath}
\usepackage{stmaryrd}
\usepackage{latexsym}
\usepackage[all]{xy}
\numberwithin{equation}{section}

\usepackage{amsthm}

\newtheorem{theorem}{Theorem}[section]
\newtheorem{proposition}[theorem]{Proposition}
\newtheorem{lemma}[theorem]{Lemma}
\newtheorem{corollary}[theorem]{Corollary}

\theoremstyle{definition}

\newtheorem{definition}[theorem]{Definition}

\newtheorem{example}[theorem]{Example}

\newtheorem{remark}[theorem]{Remark}

\DeclareMathOperator{\id}{id}

\newcommand{\cc}[1]{\overline{#1}}
\newcommand{\CC}{\mathbb{C}}
\newcommand{\RR}{\mathbb{R}}
\newcommand{\xto}{\xrightarrow}
\newcommand{\xfrom}{\xleftarrow}
\newcommand{\diese}{^{\sharp}}
\newcommand{\XX}{\mathfrak{X}}
\newcommand{\OO}{\Omega}

\definecolor{refkey}{gray}{.625}
\definecolor{labelkey}{gray}{.625}

\usepackage{geometry}
\begin{document}

\title{From Hypercomplex to Holomorphic Symplectic Structures}
\author{Wei Hong}
\address{Department of Mathematics, Penn State University}
\email{hong{\textunderscore}w@math.psu.edu}
\author{Mathieu Sti\'enon}
\address{Department of Mathematics, Penn State University}
\email{stienon@math.psu.edu}
\maketitle

\begin{abstract} The notions of holomorphic symplectic structures and hypercomplex structures on Courant algebroids are introduced and then proved to be equivalent. These generalize hypercomplex triples and holomorphic symplectic 2-forms on manifolds respectively. Basic properties of such structures are established.
\end{abstract}

\tableofcontents

\section{Introduction}

This paper is an extension of~\cite{Stienon}.
Here, we make the case that, when seen in the framework of Courant algebroids,
hypercomplex structures and holomorphic symplectic structures are one and the same concept.

A \emph{hypercomplex manifold} is a smooth manifold $M$ endowed with three complex structures $i$, $j$, $k$
(regarded as endomorphisms of the tangent bundle of $M$) that satisfy the quaternionic relations $i^2=j^2=k^2=ijk=-1$.
A characteristic feature of hypercomplex manifolds discovered by Obata early on is the existence of a unique
torsion-free connection $\nabla$ that satisfies $\nabla i=\nabla j=\nabla k=0$ ~\cite{Obata}.
Hypercomplex manifolds have been the subject of much attention in the past.
Noteworthy are the constructions of left-invariant hypercomplex structures on compact Lie groups
and homogeneous spaces due to Spindel, Sevrin, Troos \& Van~Proeyen (in 1988) and also to Joyce (in 1992).
Moreover, important examples of hypercomplex manifolds arose in mathematical physics
in the form of hyper-K\"ahler manifolds.\footnote{Hyper-K\"ahler manifolds are hypercomplex manifolds $(M;i,j,k)$
endowed with a Riemannian metric $g$ with respect to which $i$, $j$, and $k$ are covariantly constant
and mutually orthogonal.}

A \emph{holomorphic symplectic manifold} is a complex manifold $(M;j)$ endowed
with a closed nondegenerate holomorphic 2-form $\omega$.
Hyper-K\"ahler manifolds, which carry three symplectic 2-forms each of which is holomorphic
with respect to one of the three complex structures, constitute again a special subclass.

The generalized complex geometry introduced in the last decade by Hitchin~\cite{Hitchin-gcg} and Gualtieri~\cite{Gualtieri}
provides the motivation for attempting to unify hypercomplex and holomorphic symplectic structures.
A generalized complex structure on a manifold $M$ is an endomorphism $J$ of the vector bundle $TM\oplus T^*M$,
orthogonal with respect to a natural symmetric pairing, and satisfying $J^2=-1$ and $\mathcal{N}(J,J)=0$,
where $\mathcal{N}$ denotes the Nijenhuis concomitant of a pair of endomorphisms of the Courant algebroid $TM\oplus T^*M$.
A generalized complex structure on a manifold $M$ can thus be seen as a complex structure
on the corresponding (standard) Courant algebroid $TM\oplus T^*M$.
Complex structures have been defined on arbitrary Courant algebroids in a similar fashion~\cite{L-W-X,Roytenberg}.

Three new concepts are introduced in the present paper. They generalize
hypercomplex manifolds, the Obata connection, and holomorphic symplectic 2-forms
to the realm of Courant algebroids:
\begin{enumerate}
\item A \emph{hypercomplex structure on a Courant algebroid $E$} is defined as a triple of complex structures $I$, $J$, $K$ on $E$
satisfying the quaternionic relations $I^2=J^2=K^2=IJK=-1$.
Hypercomplex manifolds, holomorphic symplectic 2-forms, and hyper-Poisson manifolds provide particular examples.
The notion of  hyper-Poisson structure, also introduced in this paper, can be seen as a degenerate analogue of hyper-K\"ahler structures.
\item The analogue of the Obata connection for a Courant algebroid $E$ endowed with a hypercomplex triple $(I,J,K)$
is called a \emph{hypercomplex connection}. Though a hypercomplex connection is not itself a connection in the usual sense, its restrictions to all Dirac subbundles of $E$ stable under $I$, $J$, $K$ are torsion-free (Lie algebroid) connections.
\item A \emph{holomorphic symplectic structure on a Courant algebroid $E$ relative to a complex structure $J$ on $E$}
is a section $\Omega$ of $\wedge^2 L_J$ such that $\Omega^\sharp\overline{\Omega}^\sharp=-\id_{L_J}$ (`nondegeneracy') and
$d_{L_J^*}\Omega=0$ (`closedness'). Here $L_J$ and $L_J^*$ denote the eigenbundles of $J$.
Given a complex manifold $(M;j)$, let $J= \left(\begin{smallmatrix} j & 0 \\ 0 & -j^* \end{smallmatrix}\right)$ be the
corresponding complex structure on the standard Courant algebroid $TM\oplus T^*M$.
The holomorphic symplectic structures on $(TM\oplus T^*M;J)$ are instances of
extended Poisson structures in the sense of~\cite{C-S-X}.
\end{enumerate}

We prove the following three theorems:
\begin{enumerate}
\item A Courant algebroid endowed with a hypercomplex structure admits a unique hypercomplex connection
(see Theorems~\ref{hypercomplex-connection} and~\ref{eqv}).
\item There exists a one-to-one correspondence between the hypercomplex structures
and the holomorphic symplectic structures on a Courant algebroid (see Theorem~\ref{hyper-holosym}).
\item Given a holomorphic symplectic structure $\Omega$ on a Courant algebroid $E$ relative to a complex structure $J$ on $E$
with eigenbundles $L_J$ and $L_J^*$, the restriction of the hypercomplex connection on $E$ to any Lie subalgebroid of $L_J^*$
maximal isotropic with respect to $\Omega$ is a flat torsion-free (Lie algebroid) connection (see Theorem~\ref{connecontr}).
\end{enumerate}

Finally, given a complex Lagrangian foliation of a complex manifold $(M;j)$ endowed with a holomorphic symplectic 2-form $\omega$,
we apply the third result above to the special case in which $E=TM\oplus T^*M$,
$J=\left(\begin{smallmatrix} j & 0 \\ 0 & -j^* \end{smallmatrix}\right)$, and $\Omega=\omega+\overline{\omega}^{-1}$,
and thereby recover a connection on the Lagrangian foliation, as discovered by Behrend \& Fantechi~\cite{Kai}.

\section{Complex structures on Courant algebroids}
A \emph{Courant algebroid} (see~\cite{L-W-X,Roytenberg}) consists of a vector bundle $\pi: E\rightarrow M$,
a nondegenerate symmetric pairing
$\langle,\rangle$ on the fibers of $\pi$, a bundle map $\rho: E\rightarrow TM$ called the anchor, and an
$\mathbb{R}$-bilinear operation $\circ$ on $\Gamma(E)$ called the Dorfman bracket, which, for all
$f\in C^{\infty}(M)$ and $x,y,z\in\Gamma(E)$, satisfy the relations
\begin{gather*}
x\circ(y\circ z)=(x\circ y)\circ z+y\circ(x\circ z),\\
\rho(x\circ y)=[\rho(x),\rho(y)],\\
x\circ fy=(\rho(x)f)y+f(x\circ y),\\
x\circ y+y\circ x=2\mathcal{D}\langle x,y \rangle, \\
\mathcal{D}f\circ x=0,\\
\rho(x)\langle y,z\rangle=\langle x\circ y,z\rangle+\langle y,x\circ z\rangle,
\end{gather*}
where $\mathcal{D}: C^{\infty}(M)\rightarrow \Gamma(E)$ is the $\mathbb{R}$-linear map defined by
$\langle \mathcal{D}f,x\rangle=\frac{1}{2}\rho(x)f$.

There is a canonical isomorphism $E\xto{\Psi}E^*$ given by $\Psi(e_{1}):e_2\mapsto\langle e_{1},e_{2}\rangle$
for any $e_{1}, e_{2}\in E$ and an induced isomorphism $\wedge^{k}E\xto{\Psi}\wedge^{k}E^*$.
Sometimes, we will implicitly identify $\wedge^{k}E$ and $\wedge^{k}E^*$ in this paper.


T.~Courant described the following \emph{standard example} in~\cite{Ted}.
Given a smooth manifold $M$, the vector bundle $TM\oplus T^{*}M\rightarrow M$ carries a natural Courant algebroid structure:
the anchor map is the projection onto the tangent component, whereas the pairing and the Dorfman bracket are respectively given by \begin{gather*} \langle X+\xi,Y+\eta\rangle=\frac{1}{2}(\xi(Y)+\eta(X)) \\ \text{and}\qquad
(X+\xi)\circ(Y+\eta)=[X,Y]+(L_{X}\eta-\iota_{Y}d\xi), \end{gather*} for all $X,Y\in\XX(M)$ and $\xi,\eta\in\Omega^{1}(M)$.

Let $(E,\rho,\langle,\rangle,\circ)$ be a Courant algebroid. Given two endomorphisms $F$ and $G$ of the vector bundle $E$,
their \emph{Nijenhuis concomitant} $\mathcal{N}(F,G):\Gamma( E)\otimes_{\RR} \Gamma(E)\rightarrow\Gamma(E)$
is defined as
\begin{align*}
\mathcal{N}(F,G)(U,V)=&FU\circ GV-F(U\circ GV)-G(FU\circ V)+FG(U\circ V)\\
&+GU\circ FV-G(U\circ FV)-F(GU\circ V)+GF(U\circ V),
\end{align*}
where $U,V\in\Gamma(E)$ (see~\cite{Stienon}). Obviously, $\mathcal{N}(F,G)=\mathcal{N}(G,F)$.
In addition, we define an $\RR$-trilinear map
$\mathcal{N}_{F,G}:\Gamma(E)\otimes_{\RR}\Gamma(E)\otimes_{\RR} \Gamma(E)\rightarrow C^{\infty}(M)$
by \[ \mathcal{N}_{F,G}(U,V,W)=\langle\mathcal{N}(F,G)(U,V),W\rangle .\]
\begin{remark}\label{tensor1}
It is easy to verify that $\mathcal{N}(F,G)(U,fV)=f\mathcal{N}(F,G)(U,V)$ for all $f\in C^{\infty}(M)$.
But usually $\mathcal{N}(F,G)(fU,V)\neq f\mathcal{N}(F,G)(U,V)$; therefore, $\mathcal{N}(F,G)$ is not necessarily a tensor.
\end{remark}

\begin{lemma}\label{Nijenhuis condition}
Let $(E,\rho,\langle,\rangle,\circ)$ be a Courant algebroid. If $F, G$ are two skew-symmetric
endomorphisms of the vector bundle $E$ such that
\begin{equation}\label{kosmanncond}
FG+GF=\lambda\id_{E}, \quad (\lambda\in\mathbb{R}),
\end{equation}
then $\mathcal{N}_{F,G}\in\Gamma(\wedge^{3}E^{*}).$
\end{lemma}
\begin{proof}
For all $U,V,W\in\Gamma(E)$, we have
\begin{align*}
&\mathcal{N}(F,G)(U,V)+\mathcal{N}(F,G)(V,U)\\
=&2\mathcal{D}\langle FU,GV\rangle-2F\mathcal{D}\langle U,GV\rangle-2G\mathcal{D}\langle FU,V\rangle
+2FG\mathcal{D}\langle U,V\rangle\\
&+2\mathcal{D}\langle GU,FV\rangle-2G\mathcal{D}\langle U,FV\rangle-2F\mathcal{D}\langle GU,V\rangle
+2GF\mathcal{D}\langle U,V\rangle\\
=&-2\mathcal{D}\langle (FG+GF)U,V\rangle+2(FG+GF)\mathcal{D}\langle U,V\rangle\\
=&-2\mathcal{D}\langle \lambda U,V\rangle+2\lambda \mathcal{D}\langle U,V\rangle\\
=&0.
\end{align*}
Similarly, by a straightforward computation, we prove that
\[\langle\mathcal{N}(F,G)(U,V),W\rangle+\langle\mathcal{N}(F,G)(U,W),V\rangle=0.\]
This completes the proof.
\end{proof}
\begin{remark}
In the case $F=G$, Lemma~\ref{Nijenhuis condition} was proved by Kosmann-Schwarzbach in \cite{Kosmann-Schwarzbach}.
\end{remark}

\begin{lemma}\label{Poisson}
Let $(E,\rho,\langle,\rangle,\circ)$ be a Courant algebroid over a manifold $M$, $F$ be a skew-symmetric endomorphism
of the vector bundle $E$, and $\pi_{F}\in\Gamma(\wedge^{2}TM)$ be the bivector field defined by
\begin{equation}\label{piCourant}
\pi_{F}(df, dg)=\langle F\mathcal{D}f, \mathcal{D}g\rangle
\end{equation}
for all $f,g\in C^{\infty}(M)$.
Set $\{f, g\}=\pi_{F}(df,dg)=\langle F\mathcal{D}f,\mathcal{D}g\rangle$. Then
\[\{\{f,g\},h\}+\{\{g,h\},f\}+\{\{h,f\},g\}=-\frac{1}{4}\mathcal{N}_{F,F}(\mathcal{D}f,\mathcal{D}g,\mathcal{D}h) ,\]
which shows that $\pi_{F}$ is a Poisson bivector field on $M$ if $\mathcal{N}(F,F)=0$.
\end{lemma}
\begin{proof}
For all $f,g,h\in C^{\infty}(M)$, we have
\begin{align*}
&\{\{f,g\},h\}=\langle F\mathcal{D}\langle F\mathcal{D}f,\mathcal{D}g\rangle,\mathcal{D}h\rangle\\
=&\langle\frac{1}{2}F(F\mathcal{D}f\circ \mathcal{D}g+\mathcal{D}g\circ F\mathcal{D}f),\mathcal{D}h\rangle
=\frac{1}{2}\langle F(F\mathcal{D}f\circ \mathcal{D}g),\mathcal{D}h\rangle;\\
&\\
&\{\{g,h\},f\}=\langle F\mathcal{D}\langle F\mathcal{D}g,\mathcal{D}h\rangle,\mathcal{D}f\rangle
=\langle \mathcal{D}\langle F\mathcal{D}g,\mathcal{D}h\rangle,-F\mathcal{D}f\rangle\\
=&-\frac{1}{2}\rho(F\mathcal{D}f)\langle F\mathcal{D}g,\mathcal{D}h\rangle
=-\frac{1}{2}\langle F\mathcal{D}f\circ F\mathcal{D}g,\mathcal{D}h\rangle
-\frac{1}{2}\langle F\mathcal{D}g, F\mathcal{D}f\circ \mathcal{D}h\rangle;\\
&\\
&\{\{h,f\},g\}=\langle F\mathcal{D}\langle F\mathcal{D}h,\mathcal{D}f\rangle,\mathcal{D}g\rangle
=\langle \mathcal{D}\langle \mathcal{D}h,F\mathcal{D}f\rangle,F\mathcal{D}g\rangle\\
=&\frac{1}{2}\langle \mathcal{D}h\circ F\mathcal{D}f+F\mathcal{D}f\circ \mathcal{D}h,F\mathcal{D}g\rangle
=\frac{1}{2}\langle F\mathcal{D}g, F\mathcal{D}f\circ \mathcal{D}h\rangle.
\end{align*}
Therefore,
\[\{\{f,g\},h\}+\{\{g,h\},f\}+\{\{h,f\},g\}=\frac{1}{2}\langle F(F\mathcal{D}f\circ\mathcal{D}g),\mathcal{D}h\rangle
-\frac{1}{2}\langle F\mathcal{D}f\circ F\mathcal{D}g,\mathcal{D}h\rangle.\]
On the other hand,
\begin{align*}
&\mathcal{N}_{F,F}(\mathcal{D}f,\mathcal{D}g,\mathcal{D}h)
=\langle\mathcal{N}(F,F)(\mathcal{D}f,\mathcal{D}g),\mathcal{D}h\rangle\\
=&\langle 2F\mathcal{D}f\circ F\mathcal{D}g-2F(F\mathcal{D}f\circ \mathcal{D}g)
-2F(\mathcal{D}f\circ F\mathcal{D}g)-2F^{2}(\mathcal{D}f\circ \mathcal{D}g) ,\mathcal{D}h\rangle\\
=&\langle 2F\mathcal{D}f\circ F\mathcal{D}g-2F(F\mathcal{D}f\circ \mathcal{D}g),\mathcal{D}h\rangle.
\end{align*}
This completes the proof.
\end{proof}

\begin{definition}
An \emph{almost complex structure on a Courant algebroid} $(E,\rho,\langle,\rangle,\circ)$
is an endomorphism $J$ of the vector bundle $E$ (i.e.\ a vector bundle map over $\id_{M}:M\rightarrow M$),
which is an orthogonal transformation with respect to the pairing $\langle,\rangle$ and satisfies $J^2=-1$.
\end{definition}

Let $J$ be an almost complex structure on a Courant algebroid $E$, and let $L_{J}$
(resp.\ $\cc{L_{J}}$) be the subbundle of $E_{\mathbb{C}}=E\otimes\mathbb{C}$
associated with the eigenvalue $i=\sqrt{-1}$ (resp.\ $-i=-\sqrt{-1}$) of $J$.
Then $L_{J}$ and $\cc{L}_{J}$ are maximal isotropic subbundles
of $E_{\mathbb{C}}$, such that
$E_{\mathbb{C}}=L_{J}\oplus\cc{L}_{J}$.

The nondegenerate symmetric pairing $\langle,\rangle$ identifies $\cc{L}_{J}$
to $L^{*}_{J}$ in a canonical way.
We will, therefore, use the symbols $\cc{L}_{J}$ and $L^{*}_{J}$ interchangeably in this paper.
\begin{definition}[\cite{Hitchin-gcg}]
A \emph{complex structure on a Courant algebroid} is an almost complex structure $J$
whose Nijenhuis concomitant $\mathcal{N}(J,J)$ vanishes---$J$ is said to be integrable.
\end{definition}

It is easy to see that an almost complex structure $J$ is integrable if and only if $L_{J}$ is involutive,
i.e.\ $\Gamma(L_{J})$ is closed under the Dorfman bracket.
The following result is an immediate consequence of Lemma~\ref{Poisson}.
\begin{corollary}[\cite{Gualtieri-2007,Barton&Stienon}]
If $J$ is a complex structure on a Courant algebroid over a smooth manifold $M$,
then $\pi_{J}$ is a Poisson bivector field on $M$.
\end{corollary}

It is known that the pair of eigenbundles $(L_{J},L^{*}_{J})$ of a complex structure $J$
on a Courant algebroid forms a Lie bialgebroid in the sense of Mackenzie \& Xu~\cite{M-X}.
The exterior differentials
\[ d_{L_{J}}: \Gamma(\wedge^{k}L^{*}_{J})\rightarrow\Gamma(\wedge^{k+1}L^{*}_{J})
\qquad \text{and} \qquad
d_{L^{*}_{J}}: \Gamma(\wedge^{k}L_{J})\rightarrow\Gamma(\wedge^{k+1}L_{J}) \]
respectively associated with the Lie algebroids $L_{J}$ and $L^{*}_{J}$
satisfy $d_{L_{J}}^{2}=0$ and $d_{L^{*}_{J}}^2=0$.

\begin{example}[\cite{Gualtieri}]\label{twistcomplex}
Let $j$ be an almost complex structure on $M$, and let $E_{\phi}=(TM\oplus T^{*}M)_{\phi}$
be the standard Courant algebroid twisted by a closed 3-form $\phi\in\Omega^{3}(M)$ (see~\cite{3-form}).
Then \[ J=\begin{pmatrix} j & 0\\ 0 & -j^{*} \end{pmatrix} \] is a  complex structure
on $E_{\phi}=(TM\oplus T^{*}M)_{\phi}$ if and only if $j$ is a complex structure on $M$, and
$\phi\in\Omega^{2,1}(M)\oplus\Omega^{1,2}(M)$,
where $\Omega^{p,q}(M)$ denotes the space of exterior differential forms of type $(p,q)$ relatively to the complex structure $j$.
These conditions hold when $M$ is a complex surface and $\phi$ is any closed 3-form on $M$.
\end{example}

\begin{lemma}[\cite{holomorphic-Lie-algebrods}]\label{holomorphicPoissonGC}
If $\pi_{1}$ is a bivector field on a smooth manifold $M$ and $j$ is an endomorphism of the tangent bundle of $M$, then
\[ J=\begin{pmatrix} j & \pi_{1} \\ 0 & -j^{*} \end{pmatrix} \] is a  complex structure on the standard Courant algebroid
$E=TM\oplus T^*M$
if and only if $j$ is a complex structure on $M$ and $\pi_{1}+\sqrt{-1}\pi_{2}$, with $\pi_{2}$ defined by
$\pi_{2}\diese=-j\pi_{1}\diese=-\pi_{1}\diese j^{*}$, is a holomorphic Poisson structure with respect to $j$.
\end{lemma}

\section{Hypercomplex structures on Courant algebroids}

\subsection{Hypercomplex structure}

\begin{definition}
An \emph{almost hypercomplex structure on a Courant algebroid} $(E,\rho,\langle,\rangle,\circ)$
is a triple $(I,J,K)$ of almost complex structures on $E$ satisfying the quaternionic relations
\[ I^2=J^2=K^2=IJK=-1 .\]
\end{definition}

This definition has a few immediate consequences. From $J^2=-1$, it follows that $\frac{1}{2}(1-iJ)$
is the projection of $E\otimes\CC=L_{J}\oplus L_{J}^{*}$ onto $L_{J}$,
whereas $\frac{1}{2}(1+iJ)$ is the projection onto $L_{J}^{*}$.
Moreover, the endomorphisms $I$, $J$, and $K$ anticommute.
Therefore, both $I$ and $K$ swap the subbundles $L_{J}$ and $L^{*}_{J}$,
whereas $J$ preserves them. Finally, the relations
\begin{gather*}
\Big(\frac{1+I}{\sqrt{2}}\Big)\Big(\frac{1+iJ}{2}\Big)=\Big(\frac{1+iK}{2}\Big)\Big(\frac{1+I}{\sqrt{2}}\Big) \\
\Big(\frac{1-I}{\sqrt{2}}\Big)\Big(\frac{1+iK}{2}\Big)=\Big(\frac{1+iJ}{2}\Big)\Big(\frac{1-I}{\sqrt{2}}\Big)
\end{gather*}
imply the inclusions
\[ \Big(\frac{1+I}{\sqrt{2}}\Big) L^{*}_{J} \subset L^{*}_{K} \qquad \text{and} \qquad
\Big(\frac{1-I}{\sqrt{2}}\Big) L^{*}_{K} \subset L^{*}_{J} .\]
Since $\Big(\frac{1-I}{\sqrt{2}}\Big)\Big(\frac{1+I}{\sqrt{2}}\Big)=1$, we obtain a pair of inverse isomorphisms:
\[ L^{*}_{J}\xto{\frac{1+I}{\sqrt{2}}}L^{*}_{K} \qquad \text{and} \qquad
L^{*}_{J}\xfrom{\frac{1-I}{\sqrt{2}}}L^{*}_{K} .\]

\begin{remark}\label{tensor}
Given an almost hypercomplex structure $(I,J,K)$ on a Courant algebroid $E$,
it is easy to see that $\mathcal{N}(I,I)$, $\mathcal{N}(I,J)$, $\mathcal{N}(I,K)$,
$\mathcal{N}(J,J)$, $\mathcal{N}(J,K)$, and $\mathcal{N}(K,K)$ are $(2,1)$-tensors,
i.e.\ vector bundle maps $E\otimes E\rightarrow E$ over $\id_M$.
We can, therefore, regard $\mathcal{N}_{I,I}$, $\mathcal{N}_{I,J}$, $\mathcal{N}_{I,K}$,
$\mathcal{N}_{J,J}$, $\mathcal{N}_{J,K}$, and $\mathcal{N}_{K,K}$
as sections of $E^{*}\otimes E^{*}\otimes E^{*}$.
\end{remark}
Lemma~\ref{Nijenhuis condition} implies
\begin{lemma}\label{skew}
If $(I,J,K)$ is an almost hypercomplex structure on a Courant algebroid $E$, then
$\mathcal{N}_{I,I},\mathcal{N}_{I,J},\mathcal{N}_{I,K},\mathcal{N}_{J,J},
\mathcal{N}_{J,K},\mathcal{N}_{K,K}\in\Gamma(\wedge^{3}E^{*})$.
\end{lemma}

\begin{definition}
A \emph{hypercomplex structure on a Courant algebroid} is an almost hypercomplex structure $(I,J,K)$
such that the Nijenhuis tensors $\mathcal{N}(I,I)$, $\mathcal{N}(I,J)$, $\mathcal{N}(I,K)$,
$\mathcal{N}(J,J)$, $\mathcal{N}(J,K)$, and $\mathcal{N}(K,K)$ vanish.
\end{definition}

\begin{proposition}\label{S2}
If $(I,J,K)$ is a hypercomplex structure on a Courant algebroid $(E,\rho,\langle,\rangle,\circ)$,
then $\lambda_{1}I+\lambda_{2}J+\lambda_{3}K$ is a complex structure on $E$
for any $\lambda_{1},\lambda_{2},\lambda_{3}\in\RR$ with $\lambda_{1}^{2}+\lambda_{2}^{2}+\lambda_{3}^{2}=1$.
\end{proposition}

\begin{corollary}
If $\pi_{I},\pi_{J},\pi_{K}$ are the bivector fields associated with a hypercomplex structure $(I,J,K)$ on a Courant algebroid
as in Lemma~\ref{Poisson}, then
\[\llbracket\pi_{\alpha},\pi_{\beta}\rrbracket=0,\quad \forall\alpha,\beta\in\{I,J,K\}.\]
\end{corollary}
\begin{proof}
For any $\lambda_{1},\lambda_{2},\lambda_{3}\in\RR$ satisfying $\lambda_{1}^{2}+ \lambda_{2}^{2}+ \lambda_{3}^{2}=1$, $\lambda_{1}\pi_{I}+\lambda_{2}\pi_{J}+\lambda_{3}\pi_{K}$
is the Poisson structure associated with the complex structure $\lambda_{1}I+\lambda_{2}J+\lambda_{3}K$ on $E$, from which the corollary immediately follows.
\end{proof}

\begin{example}
The quaternion algebra $\mathbb{H}$ can be regarded as a Courant algebroid over the one point space
with the commutator as bracket and $\{1,i,j,k\}$ as an orthonormal basis. If $I, J$, and $K$ denote
the multiplication by $i,j$, and $k$ from the left respectively, then $(I,J,K)$ is a hypercomplex structure
on the Courant algebroid $(\mathbb{H}, \langle,\rangle, [ , ])$.
\end{example}

\begin{example}\label{exampleHyercomplex}
Let $i,j,$ and $k$ be almost complex structures on a smooth manifold $M$. Then the triple
\[ I=\begin{pmatrix} i & 0 \\ 0 & -i^{*} \end{pmatrix}, \qquad
J=\begin{pmatrix} j & 0 \\ 0 & -j^{*} \end{pmatrix}, \qquad
K=\begin{pmatrix} k & 0 \\ 0 & -k^{*} \end{pmatrix} \]
is a hypercomplex structure on $TM\oplus T^{*}M$ if and only if the triple
$i,j,k$ is hypercomplex in the classical sense (see~\cite{Obata} or~\cite{Yano-Ako}).
\end{example}

\begin{example}\label{exampleHoloSymp}
Let $j$ be a complex structure on a smooth manifold $M$,
and let $\omega_{1}$ and $\omega_{2}$ be two nondegenerate 2-forms on $M$. The triple
\[ I=\begin{pmatrix} 0 & \omega_{1}^{-1} \\ -\omega_{1} & 0 \end{pmatrix}, \qquad
J=\begin{pmatrix} j & 0 \\ 0 & -j^{*} \end{pmatrix}, \qquad
K=\begin{pmatrix} 0 & \omega_{2}^{-1} \\ -\omega_{2} & 0 \end{pmatrix} \]
is a hypercomplex structure on $TM\oplus T^{*}M$ if and only if
$\omega_{1}-\sqrt{-1}\omega_{2}$ is a holomorphic symplectic 2-form on $M$. We will discuss this case in more detail in Example \ref{holomorphicSymplectic}.
\end{example}

\begin{example}
Let $(i,j,k)$ be a hypercomplex structure on a four-dimensional manifold $M$,
and let $\phi$ be a closed 3-form on $M$. Then
\[ I=\begin{pmatrix} i & 0 \\ 0 & -i^{*} \end{pmatrix}, \qquad
J=\begin{pmatrix} j & 0 \\ 0 & -j^{*} \end{pmatrix}, \qquad
K=\begin{pmatrix} k & 0 \\ 0 & -k^{*} \end{pmatrix} \]
is a hypercomplex structure on the twisted Courant algebroid $E_{\phi}=(TM\oplus T^{*}M)_{\phi}$.
\end{example}

\subsection{Hypercomplex connection}
Let us recall a classical result pertaining to hypercomplex manifolds.
\begin{theorem}[Obata connection~\cite{Obata,Yano-Ako}]
\begin{enumerate}
\item Let $M$ be a manifold endowed with a hypercomplex structure $(i, j, k)$.
There exists a unique torsion-free connection $\nabla$ on $M$ such that $$\nabla i=\nabla j=\nabla k=0 ,$$
which is given by the expression
\[\nabla_{X}Y=-\frac{1}{2}k([jY,iX]-j[Y,iX]-i[jY,X]+ji[Y,X]), \quad \forall X,Y\in\XX(M) .\]
\item Conversely, given an almost hypercomplex structure $(i, j, k)$ on $M$,
if there exists a torsion-free connection $\nabla$ on $M$ such that $\nabla i=\nabla j=\nabla k=0$,
then $(i, j, k)$ must be a hypercomplex structure on $M$.
\end{enumerate}
\end{theorem}
We will generalize this result to hypercomplex structures on Courant algebroids.
Let $(I, J, K)$ be an almost hypercomplex structure on a Courant algebroid $E$.
For all $f\in C^{\infty}(M)$ and $U,V\in\Gamma(E)$, set
\begin{equation}\label{deltaf}
\triangle_{f}(U,V)=\langle U,V\rangle \mathcal{D}f+\langle IU,V\rangle I\mathcal{D}f
+\langle JU,V\rangle J\mathcal{D}f+\langle KU,V\rangle K\mathcal{D}f.
\end{equation}
It is simple to check that
\begin{gather*} \triangle_{f}(U,IV)=I\triangle_{f}(U,V), \\
\triangle_{f}(U,JV)=J\triangle_{f}(U,V), \\
\triangle_{f}(U,KV)=K\triangle_{f}(U,V), \\
\intertext{and}
\triangle_{f}(U,V)+\triangle_{f}(V,U)=2\langle U,V\rangle \mathcal{D}f
.\end{gather*}

\begin{definition}
A \emph{hypercomplex connection} on a Courant algebroid $(E,\rho,\langle,\rangle,\circ)$
endowed with an almost hypercomplex structure $(I,J,K)$ is an $\mathbb{R}$-bilinear map
\[ \Gamma(E)\otimes\Gamma(E)\to\Gamma(E), \qquad (U,V)\mapsto \nabla_{U}V \]
such that
\begin{gather}
\nabla_{fU}V=f\nabla_{U}V \label{fnabla} \\
\intertext{and} \label{nablaf}
\nabla_{U}(fV)=(\rho(U)f)V+f(\nabla_{U}V)-\triangle_{f}(U,V)
,\end{gather}
for all $f\in C^{\infty}(M)$ and $U,V\in\Gamma(E)$.
Its torsion is given by
\begin{equation}\label{deftorsion}
T^{\nabla}(U,V)=\nabla_{U}V-\nabla_{V}U-\llbracket U,V\rrbracket
\end{equation}
and its curvature by
\begin{equation}\label{defcurvature}
R^{\nabla}(U,V)W=\nabla_{U}\nabla_{V}W-\nabla_{V}\nabla_{U}W-\nabla_{\llbracket U,V\rrbracket}W
\end{equation}
for all $U,V,W\in\Gamma(E)$, where $\llbracket U,V\rrbracket=\frac{1}{2}(U\circ V-V\circ U)$.
\end{definition}

\begin{theorem} [\cite{Stienon}]\label{hypercomplex-connection}
Let $(I,J,K)$ be a hypercomplex structure on a Courant algebroid $(E,\rho,\langle,\rangle,\circ)$.
There exists a unique hypercomplex connection $\nabla$ that satisfies
\begin{equation}\label{nablaI}
\nabla I=\nabla J=\nabla K=0,
\end{equation}
and
\begin{equation}\label{torsion}
T^{\nabla}(U,V)=I\mathcal{D}\langle U,IV\rangle+J\mathcal{D}\langle U,JV\rangle
+K\mathcal{D}\langle U,KV\rangle,\quad\forall U,V\in\Gamma(E)
.\end{equation}
It is given by
\begin{equation}\label{connection}
\nabla_{U}V=-\frac{1}{2}K(JV\circ IU-J(V\circ IU)-I(JV\circ U)+JI(V\circ U)),\quad\forall U,V\in\Gamma(E)
.\end{equation}
\end{theorem}

The main result of this section is the following theorem.
\begin{theorem}[\cite{Stienon}]\label{eqv}
Let $(I,J,K)$ be an almost hypercomplex structure on a Courant algebroid
$(E,\rho,\langle,\rangle,\circ)$.
The following assertions are equivalent:
\begin{enumerate}
\item $\mathcal{N}(I,J)=0$;
\item $\mathcal{N}(I,I)=\mathcal{N}(J,J)=0$;
\item The triple $(I,J,K)$ is a hypercomplex structure: all six Nijenhuis tensors
$\mathcal{N}(I,I)$, $\mathcal{N}(I,J)$, $\mathcal{N}(I,K)$,
$\mathcal{N}(J,J)$, $\mathcal{N}(J,K)$, and $\mathcal{N}(K,K)$ vanish.
\item There exists a unique hypercomplex connection $\nabla$ that satisfies
\begin{equation}\label{nablaI}
\nabla I=\nabla J=\nabla K=0,
\end{equation}
and, for all $U,V\in\Gamma(E)$,
\begin{equation}\label{torsion}
T^{\nabla}(U,V)=I\mathcal{D}\langle U,IV\rangle+J\mathcal{D}\langle U,JV\rangle+K\mathcal{D}\langle U,KV\rangle.
\end{equation}
\end{enumerate}
\end{theorem}

As an application of Theorem~\ref{eqv} for Example~\ref{exampleHoloSymp}, we have the following corollary.
\begin{corollary}
Let $\omega_{1}$ and $\omega_{2}$ be two nondegenerate forms on a manifold $M$
endowed with an almost complex structure $j$ such that $\omega_{2}\diese=j^{*}\omega_{1}\diese=\omega_{1}\diese j$.
Then any one of the following assertions is a consequence of the other two:
\begin{itemize}
\item[(1)] $d\omega_{1}=0$,
\item[(2)] $d\omega_{2}=0$,
\item[(3)] $j$ is integrable.
\end{itemize}
\end{corollary}
From this corollary, we immediately obtain the following standard result as given by Hitchin:
\begin{lemma}[\cite{Hitchin-minimal-surface}]
Let $g$ be a Riemannian metric on a smooth manifold with skew-adjoint endomorphisms $i,j,$ and $k$
of the tangent bundle satisfying the quaternionic conditions.
Then $g$ is hyper-K\"{a}hler if and only if the corresponding 2-forms $\omega_{1},\omega_{2},\omega_{3}$ are closed.
(The 2-forms $\omega_{1},\omega_{2},\omega_{3}$ are related to the endomorphisms $i,j,k$ by
$\omega_{1}^{\sharp}=g^{\sharp}\circ i$, $\omega_{2}^{\sharp}=g^{\sharp}\circ j$,
$\omega_{3}^{\sharp}=g^{\sharp}\circ k$.)
\end{lemma}

The isotropic, involutive subbundles of a Courant algebroid are necessarily Lie algebroids. Those of maximal rank are called \emph{Dirac structures}.

Let $L$ be a Lie algebroid with anchor $\rho$ and  let $V$ be a vector bundle both over the same smooth manifold $M$.
An $L$-connection on $V$ is a bilinear map $\nabla:\Gamma(L)\times\Gamma(V)\mapsto\Gamma(V)$ that satisfies
\begin{align*}
\nabla_{fX}v&=f\nabla_{X}v,\\
\nabla_{X}fv&=f\nabla_{X}v+(\rho(X)f)v,
\end{align*}
for all $X\in\Gamma(L)$, $v\in\Gamma(V)$ and $f\in C^{\infty}(M)$.

The following lemma is a direct consequence of Theorem~\ref{eqv} and Equations~\eqref{deltaf} and~\eqref{fnabla}.
\begin{lemma}
Let $\nabla$ denote the hypercomplex connection defined by Equation~\eqref{connection}
on a Courant algebroid $(E,\rho,\langle,\rangle,\circ)$ endowed with a hypercomplex triple $(I,J,K)$.
If $L$ is an isotropic, involutive subbundle of $E$ stable under $I$, $J$, and $K$,
then $\nabla$ induces a torsion-free $L$-connection on $L$.
\end{lemma}

\begin{example}\label{ExampleHypercomplexFoliation}
Let $(M; i, j, k)$ be a hypercomplex manifold;  let $(I, J, K)$ be the corresponding hypercomplex structure
on the standard Courant algebroid $TM\oplus T^{*}M$ as in Example~\ref{exampleHyercomplex}; and
let $\mathcal{F}=TS$ be the integrable distribution corresponding to a foliation $S$. The Dirac subbundle
$L=\mathcal{F}\oplus\mathcal{F}^{\perp}$ is stable under $I, J, K$ if and only if $\mathcal{F}$ is stable under $i, j, k$.
In this situation, the hypercomplex connection $\nabla$ defined in Equation~\eqref{connection} defines
a torsion-free $L$-connection on $L$ such that $\nabla I=\nabla J=\nabla K=0$.
If we take $\mathcal{F}=TM$, then we get the Obata connection.
\end{example}

\begin{example}\label{ExampleLagrangianFoliation}
Let $(M, j, \omega)$ be a holomorphic symplectic manifold with the complex structure $j$
and the holomorphic symplectic form $\omega=\omega_{1}-\sqrt{-1}\omega_{2}$;
let $I, J, K$ be the hypercomplex structure on the standard Courant algebroid $TM\oplus T^{*}M$
as in Example~\ref{exampleHoloSymp}; and let $\mathcal{F}=TS$ be the integrable distribution of a foliation $S$.
Then the Dirac subbundle $L=\mathcal{F}\oplus\mathcal{F}^{\perp}$ is stable under $I, J, K$
if and only if $S$ is a complex Lagrangian foliation of $(M;j,\omega)$. In this situation,
the hypercomplex connection $\nabla$ defined in Equation~\eqref{connection}
defines a torsion-free $L$-connection on $L$ that satisfies $\nabla I=\nabla J=\nabla K=0$.
Explicitly, $\forall X,Y\in\Gamma(\mathcal{F})$, the connection can be written as
\[\nabla_{X}Y=\frac{1}{2}(\omega_{2}^{-1})\diese\big((\mathcal{L}_{jY}\omega_{1})\diese X
+j^{*}((\mathcal{L}_{Y}\omega_{1})\diese X)\big).\]
This torsion-free connection on $TS$ appears in Behrend and Fantechi's work
on the Donaldson-Thomas invariants~\cite{Kai}.
We will return to it in Corollary~\ref{LagrangianFoliation}.
\end{example}

\subsection{Proofs of Theorems~\ref{hypercomplex-connection} and~\ref{eqv}}

\begin{proof}[Proof of Theorem~\ref{hypercomplex-connection}]
Let $\nabla:\Gamma(E)\otimes\Gamma(E)\rightarrow\Gamma(E)$ be the bilinear map as defined in Equation~\eqref{connection}.

\textbf{(1)}\quad We will prove that $\nabla I=\nabla J=\nabla K=0$.

By the properties of Courant algebroid and almost hypercomplex conditions, we can verify that $\nabla$  is
indeed  a hypercomplex connection.

Given that for all $U\in\Gamma(E), Y\in\Gamma(L_{J})$,
\[\nabla_{U}Y=-I\frac{iJ+1}{2}(Y\circ IU-I(Y\circ U))=-\frac{1-iJ}{2}I(Y\circ IU-I(Y\circ U))\in\Gamma(L_{J}),\]
we have $\nabla_{U}JY=i\nabla_{U}Y=J\nabla_{U}Y$.  Similarly, we have $\nabla_{U}J\xi=J\nabla_{U}\xi$, $\forall U\in\Gamma(E), \xi\in\Gamma(L_{J}^{*})$. Therefore,  $\nabla J=0$.

 A simple computation shows that, for all $U\in\Gamma(E), Y\in\Gamma(L_{J})$,
\begin{align*}
&I\nabla_{U}Y-\nabla_{U}IY\\
=&\frac{iJ+1}{2}(Y\circ IU-I(Y\circ U))+\frac{iJ+1}{2}I(IY\circ IU-I(IY\circ U))\\
=&\frac{iJ+1}{4}I\mathcal{N}(I,I)(Y,U)=0.
\end{align*}
Hence, we have $I\nabla_{U}Y=\nabla_{U}IY$. Similarly, we have $I\nabla_{U}\xi=\nabla_{U}I\xi, \forall U\in\Gamma(E), \xi\in\Gamma(L^{*}_{J})$. Therefore, $\nabla I=0$.

As $K=IJ$, we have $\nabla_{U} K=(\nabla_{U} I)J+I\nabla_{U} J=0, \forall U\in\Gamma(E).$
Thus, $\nabla K=0$.

\textbf{(2)}\quad Next we will prove Equation~\eqref{torsion}.
We claim that if $(I,J,K)$ is an almost hypercomplex structure
on a Courant algebroid $(E,\rho,\langle,\rangle,\circ)$, then, for all $U,V\in\Gamma(E)$,
\begin{equation}\label{torsion1}
T^{\nabla}(U,V)-\frac{1}{2}K\mathcal{N}(I,J)(U,V)=I\mathcal{D}\langle U,IV\rangle
+J\mathcal{D}\langle U,JV\rangle+K\mathcal{D}\langle U,KV\rangle.
\end{equation}
For all $U,V\in\Gamma(E)$, by the Courant algebroid properties and almost hypercomplex conditions, we have
\begin{align*}
&T^{\nabla}(U,V)-\frac{1}{2}K\mathcal{N}(I,J)(U,V)\\
=&\nabla_{U}V-\nabla_{V}U-\llbracket U,V\rrbracket -\frac{1}{2}K\mathcal{N}(I,J)(U,V)\\
=&-\frac{1}{2}K(JV\circ IU-J(V\circ IU)-I(JV\circ U)+JI(V\circ U))\\
&+\frac{1}{2}K(JU\circ IV-J(U\circ IV)-I(JU\circ V)+JI(U\circ V))-\frac{1}{2}(U\circ V-V\circ U)\\
&-\frac{1}{2}K(IU\circ JV-I(U\circ JV)-J(IU\circ V)+JU\circ IV-J(U\circ IV)-I(JU\circ V))\\
=&-\frac{1}{2}K(JV\circ IU-J(V\circ IU)-I(JV\circ U)+IU\circ JV-I(U\circ JV)-J(IU\circ V))\\
=&-K(\mathcal{D}\langle IU,JV\rangle-I\mathcal{D}\langle U,JV\rangle-J\mathcal{D}\langle IU,V\rangle)\\
=&I\mathcal{D}\langle U,IV\rangle+J\mathcal{D}\langle U,JV\rangle+K\mathcal{D}\langle U,KV\rangle.
\end{align*}
As $\mathcal{N}(I,J)=0$, Equation~\eqref{torsion} thus follows.

\textbf{(3)}\quad We will now prove the uniqueness of the hypercomplex connection
that satisfies Equations~\eqref{nablaI} and~\eqref{torsion}.
Assume there exist two hypercomplex connections $\nabla^{1}$ and $\nabla^{2}$ that satisfy Equations~\eqref{nablaI}
and~\eqref{torsion}. For all $U,V\in\Gamma(E)$, set $\Xi(U,V)=\nabla^{1}_{U}V-\nabla^{2}_{U}V$.
It follows from Equation~\eqref{nablaI} that
\[ \Xi(U,IV)=I\Xi(U,V), \quad \Xi(U,JV)=J\Xi(U,V), \quad \Xi(U,KV)=K\Xi(U,V), \]
and from Equation~\eqref{torsion} that $\Xi(U,V)=\Xi(V,U)$. Therefore,
\begin{align*}
&K\Xi(U,U)=IJ\Xi(U,U)=I\Xi(U,JU)=I\Xi(JU,U)=\Xi(JU,IU)\\
=&\Xi(IU,JU)=J\Xi(IU,U)=J\Xi(U,IU)=JI\Xi(U,U)=-K\Xi(U,U).
\end{align*}
Hence, $\Xi(U,U)=0$ for all $U\in\Gamma(E)$. Consequently,
\[ \Xi(U,V)=\frac{1}{2}(\Xi(U+V,U+V)-\Xi(U,U)-\Xi(V,V))=0 \]
for all $U,V\in\Gamma(E)$. Thus, the uniqueness of the hypercomplex connection satisfying Equations~\eqref{nablaI} and~\eqref{torsion} is established.
\end{proof}

Theorem~\ref{eqv} follows from Theorem~\ref{hypercomplex-connection} and the next three lemmas.

\begin{lemma}
 Let $(I,J,K)$ be an almost hypercomplex structure on a Courant algebroid $(E,\rho,\langle,\rangle,\circ)$.
Assume that there exists a hypercomplex connection $\nabla$ satisfying $\nabla I=\nabla J=\nabla K=0$ and
$T^{\nabla}(U,V)=I\mathcal{D}\langle U,IV\rangle+J\mathcal{D}\langle U,JV\rangle+K\mathcal{D}\langle U,KV\rangle$
for all $U,V\in\Gamma(E)$. Then $\mathcal{N}(I,J)=0$.
\end{lemma}
\begin{proof}
Assume that there exists a hypercomplex connection $\nabla$ satisfying Equations~\eqref{nablaI} and~\eqref{torsion}.
Equation~\eqref{torsion} implies that
\begin{equation*}
U\circ V=\nabla_{U}V-\nabla_{V}U+\mathcal{D}\langle U,V\rangle
-(I\mathcal{D}\langle U,IV\rangle+J\mathcal{D}\langle U,JV\rangle+K\mathcal{D}\langle U,KV\rangle)
\end{equation*}
for all $U,V\in\Gamma(E)$. Therefore,
\[ IU\circ JV=\nabla_{IU}JV-\nabla_{JV}IU+\mathcal{D}\langle IU,JV\rangle
-(I\mathcal{D}\langle IU,IJV\rangle+J\mathcal{D}\langle IU,JJV\rangle+K\mathcal{D}\langle IU,KJV\rangle).\]
Since $\nabla I=\nabla J=0$, it follows that
\[IU\circ JV=J\nabla_{IU}V-I\nabla_{JV}U-\mathcal{D}\langle U,KV\rangle-I\mathcal{D}\langle U,JV\rangle
-J\mathcal{D}\langle U,IV\rangle+K\mathcal{D}\langle U,V\rangle.\]
Similarly, we have
\begin{align*}
&JU\circ IV=I\nabla_{JU}V-J\nabla_{IV}U+\mathcal{D}\langle U,KV\rangle-I\mathcal{D}\langle U,JV\rangle
-J\mathcal{D}\langle U,IV\rangle-K\mathcal{D}\langle U,V\rangle;\\
&U\circ JV=J\nabla_{U}V-\nabla_{JV}U+\mathcal{D}\langle U,JV\rangle-I\mathcal{D}\langle U,KV\rangle
+J\mathcal{D}\langle U,V\rangle+K\mathcal{D}\langle U,IV\rangle;\\
&JU\circ V=\nabla_{JU}V-J\nabla_{V}U-\mathcal{D}\langle U,JV\rangle-I\mathcal{D}\langle U,KV\rangle
-J\mathcal{D}\langle U,V\rangle+K\mathcal{D}\langle U,IV\rangle;\\
&U\circ IV=I\nabla_{U}V-\nabla_{IV}U+\mathcal{D}\langle U,IV\rangle+I\mathcal{D}\langle U,V\rangle
+J\mathcal{D}\langle U,KV\rangle-K\mathcal{D}\langle U,JV\rangle;\\
&IU\circ V=\nabla_{IU}V-I\nabla_{V}U-\mathcal{D}\langle U,IV\rangle-I\mathcal{D}\langle U,V\rangle
+J\mathcal{D}\langle U,KV\rangle-K\mathcal{D}\langle U,JV\rangle.
\end{align*}
A simple computation shows that
\[ \mathcal{N}(I,J)(U,V)=IU\circ JV-I(U\circ JV)-J(IU\circ V)+JU\circ IV-J(U\circ IV)-I(JU\circ V)=0 \]
for all $U,V\in\Gamma(E)$.
Hence, $\mathcal{N}(I,J)=0$.
\end{proof}

\begin{lemma}
Let $(I,J,K)$ be an almost hypercomplex structure on a Courant algebroid $(E,\rho,\langle,\rangle,\circ)$.
If $\mathcal{N}(I,J)=0$, then $\mathcal{N}(I,I)=\mathcal{N}(J,J)=0$.
\end{lemma}
\begin{proof}
Since for all $X,Y\in\Gamma(L_{J})$,
\begin{align*}
&\mathcal{N}(I,J)(X,Y)
= IX\circ JY-I(X\circ JY)-J(IX\circ Y)+JX\circ IY-J(X\circ IY)-I(JX\circ Y)\\
=&(i-J)(IX\circ Y+X\circ IY)-2iI(X\circ Y),
\end{align*}
we have $X\circ Y=-\frac{1-iJ}{2}I(IX\circ Y+X\circ IY)\in\Gamma(L_{J} )$ from $N(I, J)=0$.
Thus, $L_{J}$ is involutive, or equivalently $\mathcal{N}(J,J)=0$.
Similarly, $\mathcal{N}(I,I)=0$.
\end{proof}

\begin{lemma}
Let $(I,J,K)$ be an almost hypercomplex structure on a Courant algebroid $(E,\rho,\langle,\rangle,\circ)$.
If $\mathcal{N}(I,I)=\mathcal{N}(J, J)=0$, then  all six Nijenhuis tensors
$\mathcal{N}(I,I)$, $\mathcal{N}(I,J)$, $\mathcal{N}(I,K)$,
$\mathcal{N}(J,J)$, $\mathcal{N}(J,K)$, and $\mathcal{N}(K,K)$ vanish.
\end{lemma}
\begin{proof}
First, we will prove $\mathcal{N}(I,J)=0$ and $\mathcal{N}(K,K)=0$ by checking that
$\mathcal{N}(I,J)(X,Y)=\mathcal{N}(I,J)(\xi,\eta)=\mathcal{N}(I,J)(X,\xi)=0$ and $\mathcal{N}( K,K)(X,Y)=\mathcal{N}( K,K)(\xi,\eta)=\mathcal{N}( K,K)(X,\xi)=0$ for all $X,Y\in\Gamma(L_J), \xi,\eta\in\Gamma(L_{J}^*)$.

Let
\begin{align*}
&P_{1}(X, Y)=IX\circ IY-I\frac{1-iJ}{2}(X\circ IY)-I\frac{1-iJ}{2}(IX\circ Y),\\
&P_{2}(X, Y)=I\frac{1+iJ}{2}(X\circ IY)+I\frac{1+iJ}{2}(IX\circ Y)+X\circ Y.
\end{align*}
Then we have $P_{1}(X,Y)\in\Gamma(L^{*}_{J})$ and $P_{2}(X, Y)\in\Gamma(L_{J})$, since $I$ swaps $L_{J}$ and $L^{*}_{J}$ and since both $L_{J},L^{*}_{J}$ are involutive.  A simple computation shows that
$\frac{1}{2}\mathcal{N}(I,I)(X,Y)=P_{1}(X,Y)-P_{2}(X, Y).$
Thus, we have $P_{1}(X,Y)=P_{2}(X,Y)=0.$
As a consequence,  by a straightforward computation,  we get
\begin{align}
&\mathcal{N}(K, K)(X,Y)=-2P_{1}(X, Y)-2P_{2}(X,Y)=0,\label{nkkxy}\\
&\mathcal{N}(I,J)(X,Y)=-2iIP_{2}(X, Y)=0.\label{nijxy}
\end{align}
Similarly,
\begin{equation}\label{nijkkxieta}
\mathcal{N}(I, J)(\xi,\eta)=\mathcal{N}(K,K)(\xi,\eta)=0.
\end{equation}
Moreover, we have
\begin{equation}\label{nijxxi}
\mathcal{N}(I,J)(X,\xi)=(-i-J)(IX\circ\xi)+(i-J)(X\circ I\xi)=0.
\end{equation}
Now $\mathcal{N}(I,J)=0$ follows from Equations \eqref{nijxy}, \eqref{nijkkxieta}, and \eqref{nijxxi}.
On the other hand,
\[\frac{1}{2}\mathcal{N}(K,K)(X,\xi)=\frac{1}{2}\mathcal{N}(IJ,IJ)(X,\xi)
= IX\circ I\xi+iIJ(X\circ I\xi)-iIJ(IX\circ\xi)-X\circ\xi.\]
As $I$ swaps $L_J$ and $L_{J}^{*}$, and as both $L_J, L_{J}^{*}$ are involutive, we have
\begin{equation}\label{nkkxxi}
\frac{1}{2}\mathcal{N}(K,K)(X,\xi)=IX\circ I\xi-I(X\circ I\xi)-I(IX\circ\xi)-X\circ\xi
=\frac{1}{2}\mathcal{N}(I,I)(X,\xi)=0.
\end{equation}
It follows from Equations \eqref{nkkxy}, \eqref{nijkkxieta} and \eqref{nkkxxi} that $\mathcal{N}(K,K)=0$.

Similarly, $\mathcal{N}(I,K)=0$, since $\mathcal{N}(I, I)=\mathcal{N}(K,K)=0$.
And, $\mathcal{N}(J,K)=0$, since $\mathcal{N}(J, J)=\mathcal{N}(K,K)=0$.

This completes the proof.
\end{proof}

\section{Holomorphic symplectic structures on Courant algebroids}
\subsection{Holomorphic symplectic structure on an arbitrary Courant algebroid}
Let $(E,\rho,\langle,\rangle,\circ)$ be a Courant algebroid endowed with a complex structure $J$.
The nondegenerate pairing $\langle,\rangle$ induces a bijection between sections of $\otimes^2 E_{\CC}^*$
and endomorphisms of $E_{\CC}$, which associates an endomorphism $\Omega\diese$ of $E_{\CC}$
with a section $\Omega$ of $\otimes^2 E_{\CC}^*$:
\[ \Omega(U,V)=\langle\Omega\diese U,V\rangle,\qquad\forall U,V\in \Gamma(E_{\CC}) .\]
The complex vector bundle $E_{\CC}$ decomposes as the direct sum
$L_{J}\oplus\cc{L}_{J}\cong L_{J}\oplus L_{J}^*$,
where we identify $\cc{L}_{J}$ and $L_{J}^*$, therefore,
an endomorphism $\Omega\diese$ of $E_{\CC}$ skew-symmetric w.r.t.\ the pairing $\langle,\rangle$
corresponds to a section $\Omega$ of
\begin{equation}\label{Ec}
\wedge^{2}E^*_{\CC}=\wedge^{2}L_{J}
\oplus(L_{J}\wedge L_{J}^{*})\oplus\wedge^{2}L_{J}^{*} .
\end{equation}
In particular, whenever $\Omega\diese(L^{*}_{J})=L_{J}$ and $\Omega\diese(L_{J})=0$,
the components of $\Omega$ in $L_{J}\wedge L_{J}^{*}$ and $\wedge^{2}L_{J}^{*}$ vanish.
We can, therefore, consider $\Omega$ to be a section of $\wedge^{2}L_{J}$.

\begin{lemma}\label{algcond}
\begin{enumerate}
\item Let $(I,J,K)$ be an almost hypercomplex structure on a Courant algebroid $(E,\rho,\langle,\rangle,\circ)$.
The endomorphism $\Omega\diese=\frac{I+iK}{2}$ of $E_{\CC}$ determines a section of
$\wedge^{2}L_{J}\subset\wedge^{2}E^*_{\CC}$ and satisfies
\begin{equation}\label{Omega1}
\Omega^{\sharp}\cc{\Omega}^{\sharp}+\cc{\Omega}^{\sharp}\Omega^{\sharp}=-\id_{E_{\CC}}.
\end{equation}
\item
Conversely, given an almost complex structure $J$ on a Courant algebroid $(E,\rho,\langle,\rangle,\circ)$
and a section $\Omega$ of $\wedge^{2}L_{J}\subset\wedge^{2}E_{\CC}^*$
satisfying $\Omega^{\sharp}\cc{\Omega}^{\sharp}+\cc{\Omega}^{\sharp}\Omega^{\sharp}=-\id_{E_{\CC}}$,
the triple $\big(I=\Omega\diese+\cc{\Omega}\diese$, $J$, $K=(-i)(\Omega\diese-\cc{\Omega}\diese)\big)$
is an almost hypercomplex structure on $E$.
\end{enumerate}
\end{lemma}

\begin{proof}
\quad\textbf{(1)}\quad  The fact that the bilinear form $\Omega$ is skew-symmetric is a direct consequence of the skew-symmetry of $I, K$.
Thus, $\Omega$ is a section of \[\wedge^{2}E^*_{\CC}\cong\wedge^{2}L_{J}
\oplus(L_{J}\wedge L_{J}^{*})\oplus\wedge^{2}L_{J}^{*} .\]
Since $\Omega\diese=I\big(\frac{1+iJ}{2}\big)$ and $I$ swaps $L_{J}$ and $L_{J}^{*}$, we have
$\Omega\diese(L^{*}_{J})=L_{J}$, $\Omega\diese(L_{J})=0$,
$\cc{\Omega}\diese(L_{J})=L_{J}^{*}$, and $\cc{\Omega}\diese(L_{J}^{*})=0$.
Therefore, $\Omega\in\Gamma(\wedge^{2}L_{J})$ and $\cc{\Omega}\in\Gamma(\wedge^{2}L_{J}^*)$.
Finally, we have
\[\Omega^{\sharp}\cc{\Omega}^{\sharp}+\cc{\Omega}^{\sharp}\Omega^{\sharp}=\Big(\frac{I+iK}{2}\Big)\Big(\frac{I-iK}{2}\Big)+\Big(\frac{I-iK}{2}\Big)\Big(\frac{I+iK}{2}\Big)=-1.\]

\textbf{(2)}\quad  As $\Omega\in\Gamma(\wedge^{2}L_{J})$, we have $(\Omega^{\sharp})^2=(\cc{\Omega}^{\sharp})^2=0$. It follows that
$ I^{2}=K^{2}=\Omega^{\sharp}\cc{\Omega}^{\sharp}+\cc{\Omega}^{\sharp}\Omega^{\sharp}=-1 .$
For all $X\in\Gamma(L_{J}), \xi\in\Gamma(L_{J}^*)$, we have $\Omega\diese X=0,$ and $\cc{\Omega}\diese\xi=0$; consequently, we have
\[IX=\cc{\Omega}\diese X,~I\xi=\Omega\diese\xi,~KX=i\cc{\Omega}\diese X,~K\xi=-i\Omega\diese\xi.\]
Thus,
\[( IJ)(X+\xi)=iIX-iI\xi=i\cc{\Omega}^{\sharp}X-i\Omega^{\sharp}\xi=K(X+\xi) .\]
Therefore, we have $K=IJ$.  The orthogonality of $I$ with respect to the pairing $\langle,\rangle$ follows from $I^2=-1$ and the skew-symmetry of $I=\Omega\diese+\cc{\Omega}\diese$, while, the orthogonality of $K$ follows from the identity $K=IJ$.
\end{proof}

\begin{remark}\label{Oginvert}
Equation~\eqref{Omega1} means that $-\cc{\Omega}^{\sharp}:L_{J}\rightarrow L_{J}^{*}$
is the inverse map of $\Omega^{\sharp}:L^{*}_{J}\rightarrow L_{J}$.
Thus, Equation~\eqref{Omega1} can be regarded as a nondegeneracy condition on $\Omega$.
\end{remark}

The following lemma, which can easily be verified, will be needed later on.
\begin{lemma} \label{screwcap}
Assume that $(I,J,K)$ is an almost hypercomplex structure on a Courant algebroid and set $\Omega\diese=\tfrac{1}{2}(I+iK)$. Then,
\begin{enumerate}
\item $e\in L_{J}^{*}$ if and only if $\Omega^{\sharp}e=Ie=iKe$,
\item $e\in L_{J}$ if and only if $\cc{\Omega}^{\sharp}e=Ie=-iKe$.
\end{enumerate}
\end{lemma}

We are now ready to introduce the following definition.
\begin{definition}
Let $J$ be a complex structure on a Courant algebroid $E$.
A \emph{holomorphic symplectic structure on $E$ with respect to $J$} is a section $\Omega$ of $\wedge^2 L_{J}$ satisfying
\[d_{L_{J}^*}\Omega=0 \qquad\text{and}\qquad \Omega\diese\cc{\Omega}\diese
+\cc{\Omega}\diese\Omega\diese=-\id_{E_{\CC}}.\]
\end{definition}

Given a hypercomplex triple $(I,J,K)$, Lemma~\ref{skew} implies that
$\mathcal{N}_{I,J},\mathcal{N}_{J,K}\in\Gamma(\wedge^{3}E^{*})$.
We now extend $\mathcal{N}_{I,J}$ and $\mathcal{N}_{J,K}$ $\CC$-linearly to 3-forms on $E_{\CC}$.

\begin{lemma}\label{hololemma}
Let $(I, J, K)$ be a hypercomplex triple on a  Courant algebroid $(E,\rho,\langle,\rangle,\circ)$, and let
$\Omega$ be the section of $\wedge^2 L_{J}$ defined by $\Omega\diese=\frac{I+\sqrt{-1}K}{2}$. Then
\begin{eqnarray}
\frac{1}{4}\mathcal{N}_{I,J}=\frac{d_{L_{J}^*}\Omega-\cc{d_{L_{J}^*}\Omega}}{2i} \label{sicko}\\
 -\frac{1}{4}\mathcal{N}_{J,K}=\frac{d_{L_{J}^*}\Omega+\cc{d_{L_{J}^*}\Omega}}{2}\label{sickoo}.
\end{eqnarray}
\end{lemma}
\begin{proof}
First, we note that $\mathcal{N}_{I,J}$ is a section of $\wedge^3L_{J}\oplus\wedge^3L_{J}^*$
because $\langle\mathcal{N}(I,J)(X,Y),\xi\rangle=0$ and $\langle\mathcal{N}(I,J)(\xi,\eta),X\rangle=0$
for all $X,Y\in\Gamma(L_{J})$  and $\xi,\eta\in\Gamma(L_{J}^{*})$.
Indeed, $\langle\mathcal{N}(I,J)(X,Y),\xi\rangle=0$ follows from 
\[\mathcal{N}(I,J)(X,Y)=(i-J)(IX\circ Y+X\circ IY)-2iI(X\circ Y)\in\Gamma(L_J^*),\]
 and $\langle\mathcal{N}(I,J)(\xi,\eta),X\rangle=0$ follows from 
 \[\mathcal{N}(I,J)(\xi,\eta)=(-i-J)(I\xi\circ \eta+\xi\circ I\eta)+2iI(\xi\circ \eta)\in\Gamma(L_J).\]
 For all $\xi,\eta,\zeta\in\Gamma(L_{J}^*)$, we have
\begin{align*}
&\langle\mathcal{N}(I,J)(\xi,\eta),\zeta\rangle=\langle(-i-J)(I\xi\circ \eta+\xi\circ I\eta)+2iI(\xi\circ \eta),\zeta\rangle\\
=&-2i\langle I\xi\circ \eta,\zeta\rangle-2i\langle \xi\circ I\eta,\zeta\rangle+2i\langle I(\xi\circ \eta),\zeta\rangle
=-2i\langle \Omega^{\sharp}\xi\circ \eta,\zeta\rangle-2i\langle \xi\circ \Omega^{\sharp}\eta,\zeta\rangle
+2i\langle \Omega^{\sharp}(\xi\circ \eta),\zeta\rangle\\
=&-2i\langle 2\mathcal{D}\langle\Omega^{\sharp}\xi,\eta\rangle,\zeta\rangle+2i\langle\eta\circ\Omega^{\sharp}\xi,\zeta\rangle-2i\langle\xi\circ\Omega^{\sharp}\eta,\zeta\rangle+2i\langle \Omega^{\sharp}(\xi\circ \eta),\zeta\rangle\\
=&-2i\rho(\zeta)\langle\Omega^{\sharp}\xi,\eta\rangle+2i\rho(\eta)\langle\Omega^{\sharp}\xi,\zeta\rangle
-2i\langle\Omega^{\sharp}\xi,\eta\circ\zeta\rangle
-2i\rho(\xi)\langle\Omega^{\sharp}\eta,\zeta\rangle
+2i\langle\Omega^{\sharp}\eta,\xi\circ\zeta\rangle+2i\langle \Omega^{\sharp}(\xi\circ \eta),\zeta\rangle\\
=&-2i\rho(\zeta)\Omega(\xi,\eta)+2i\rho(\eta)\Omega(\xi,\zeta)-2i\Omega(\xi,\eta\circ\zeta)
-2i\rho(\xi)\Omega(\eta,\zeta)+2i\Omega(\eta,\xi\circ\zeta)+2i\Omega(\xi\circ \eta,\zeta)\\
=&-2i(d_{L_{J}^{*}}\Omega)(\xi,\eta,\zeta).
\end{align*}
Now Equation~\eqref{sicko} holds, as $\cc{\mathcal{N}_{I,J}}=\mathcal{N}_{I,J}$ and $L_{J}^*=\cc{L_{J}}$.

The analogous relation
\[ -\frac{1}{4}\mathcal{N}_{J,K}=\frac{d_{L_{J}^*}\Omega+\cc{d_{L_{J}^*}\Omega}}{2} \]
can be proved in the same way.
\end{proof}

By Lemma \ref{algcond}, Lemma~\ref{hololemma} and Theorem~\ref{eqv}, we have the following theorem.

\begin{theorem}\label{hyper-holosym}
\begin{enumerate}
\item If $(I, J, K)$ is a hypercomplex triple on a Courant algebroid $E$,
then the section $\Omega$ of $\wedge^2 L_{J}$ defined by $\Omega\diese=\frac{I+\sqrt{-1}K}{2}$ is a holomorphic symplectic
structure on $E$ relative to the complex structure $J$.
\item  Let $J$ be a complex structure on a Courant algebroid $E$,
and let $\Omega\in\Gamma(\wedge^2 L_{J})$ be a holomorphic symplectic structure on $E$
relative to the complex structure $J$. Then the triple
$\big(I=\Omega\diese+\cc{\Omega}\diese, J, K=-i(\Omega\diese-\cc{\Omega}\diese)\big)$
is a hypercomplex structure on $E$.
\end{enumerate}
\end{theorem}

The pair of eigenbundles $(L_{J},L_{J}^{*})$ of a complex structure $J$ on a Courant algebroid
constitutes a  Lie bialgebroid; therefore, $(\Gamma(\wedge^{*}L_{J}), [ , ], d_{L_{J}^*}) $ is a differential graded Lie algebra, where $[ , ]$ denotes the Schouten bracket and $d_{L_{J}^*}$ denotes the Lie algebroid differential.

\begin{theorem} \label{EqvOmega}
Let $L_J$ and $L_{J}^*$ denote the eigenbundles of a complex structure $J$
on a Courant algebroid $(E,\rho,\langle,\rangle,\circ)$, and let $\Omega$ be a section of
$\wedge^2 L_{J}\subset\wedge^{2}E_{\CC}$ such that
$\Omega\diese\cc{\Omega}\diese+\cc{\Omega}\diese\Omega\diese=-1$.
The following assertions are equivalent:
\begin{enumerate}
\item $[\Omega,\Omega]=0$, where $[\cdot,\cdot]$ stands for the Schouten bracket on $\Gamma(\wedge^{\bullet}L_{J})$,
\item $d_{L_{J}^*}\Omega=0,$
\item $d_{L_{J}^*}\Omega+\frac{1}{2}[\Omega,\Omega]=0$.
\end{enumerate}
\end{theorem}

Theorem~\ref{EqvOmega} follows from Theorem~\ref{hyper-holosym}, Theorem~\ref{eqv},
 Lemma \ref{EqvOmegal1}, and Lemma \ref{EqvOmegal2}.

To prove  Lemma \ref{EqvOmegal1} and Lemma \ref{EqvOmegal2}, we need the following lemma, which is an application of Theorem 6.1 and Equations (23) and (24) in \cite{L-W-X} in the case of the Lie bialgebroid $(L_J, L_J^*)$. 
\begin{lemma}\label{EqvOmega3}
Let $L_J$ and $L_{J}^*$ denote the eigenbundles of a complex structure $J$
on a Courant algebroid $(E,\rho,\langle,\rangle,\circ)$, and let $\Omega$ be a section of
$\wedge^2 L_{J}$:
\begin{enumerate}
\item The subbundle $(1+\Omega\diese)L_{J}^{*}$ of $E_{\CC}$ is involutive if and only if $d_{L^{*}_{J}}\Omega+\frac{1}{2}[\Omega,\Omega]=0$.
\item For all $\xi,\eta\in\Gamma(L_{J}^{*})$, we have $(\frac{1}{2}[\Omega,\Omega])\diese(\xi,\eta)=\Omega^{\sharp}(\mathcal{L}_{\Omega^{\sharp}\xi}\eta-\mathcal{L}_{\Omega^{\sharp}\eta}\xi
+d_{L_{J}}\langle\xi,\Omega^{\sharp}\eta\rangle)-[\Omega^{\sharp}\xi,\Omega^{\sharp}\eta]$.
\end{enumerate}
\end{lemma}

\begin{lemma}\label{EqvOmegal1}
Given the same hypothesis as in Theorem~\ref{EqvOmega}, we have
\begin{equation*}
\frac{1}{2}[\Omega,\Omega](\xi,\eta,\zeta)=\cc{d_{L_{J}^{*}}\Omega}(\Omega\diese\xi,\Omega\diese\eta,\Omega\diese\zeta),\quad \forall \xi,\eta,\zeta\in\Gamma(L_{J}^{*}) .
\end{equation*}
\end{lemma}
\begin{proof}
According to Lemma \ref{EqvOmega3}, for all $ \xi,\eta,\zeta\in\Gamma(L_{J}^{*})$, we have
\[\frac{1}{2}[\Omega,\Omega](\xi,\eta,\zeta)=-\langle\mathcal{L}_{\Omega\diese\xi}\eta-\mathcal{L}_{\Omega\diese\eta}\xi+d_{L_{J}}\langle\xi,
\Omega\diese\eta\rangle,\Omega\diese\zeta\rangle-\langle[\Omega\diese\xi,\Omega\diese\eta],\zeta\rangle.\]
Set $X=\Omega\diese\xi$, $Y=\Omega\diese\eta$, and $Z=\Omega\diese\zeta$. Then we have $X,Y,Z\in\Gamma(L_{J})$. And,  by Lemma \ref{screwcap}, we have $\xi=-\cc{\Omega}\diese X, \eta=-\cc{\Omega}\diese Y, \zeta=-\cc{\Omega}\diese Z$. Moreover,
\begin{align*}
& \frac{1}{2}[\Omega,\Omega](\xi,\eta,\zeta)
=\langle\mathcal{L}_{X}\cc{\Omega}\diese Y-\mathcal{L}_{Y}\cc{\Omega}\diese X+d_{L_J}\langle\cc{\Omega}\diese X,Y\rangle, Z\rangle+\langle [X,Y], \cc{\Omega}\diese Z\rangle\\
=&\rho(X)\cc{\Omega}(Y,Z)-\cc{\Omega}(Y,[X,Z])-\rho(Y)\cc{\Omega}(X,Z)+\cc{\Omega}(X,[Y,Z])
+\rho(Z)\cc{\Omega}(X,Y)-\cc{\Omega}([X,Y],Z)\\
=&d_{L_{J}}\cc{\Omega}(X,Y,Z)=\overline{d_{\cc{L}_{J}}\Omega}(X,Y,Z) \\
=&\overline{d_{L_{J}^{*}}\Omega}(\Omega\diese\xi,\Omega\diese\eta,\Omega\diese\zeta).
\end{align*}
This completes the proof.
\end{proof}

\begin{lemma}\label{EqvOmegal2}
Given the same hypothesis as in Theorem~\ref{EqvOmega}, set $K=(-i)(\Omega\diese-\cc{\Omega}\diese)$. We have
 \[ \mathcal{N}(K,K)=0 \qquad  \text{if and only if} \qquad d_{L_{J}^*}\Omega+\tfrac{1}{2}[\Omega,\Omega]=0 .\]
\end{lemma}
\begin{proof}
By assumption, we have $\mathcal{N}(J,J)=0$.
By Lemma \ref{EqvOmega3}, the subbundle
\[ (1+\Omega\diese)L_{J}^*=\{\xi+\Omega\diese\xi\in E_{\CC}|\xi\in L_{J}^* \} \]
of $E_{\CC}$ is involutive if and only if $d_{L_{J}^*}\Omega+\frac{1}{2}[\Omega,\Omega]=0$.
Since $L_{K}^*=\frac{1+I}{\sqrt{2}}L_{J}^*$ and $I\xi=\Omega\diese\xi$ for all $\xi\in L_{J}^*$ (see Lemma~\ref{screwcap}),
we have $L_{K}^*=(1+\Omega\diese)L_{J}^*$.
Since the involutivity of $L_{K}^*$ is equivalent to the vanishing of the Nijenhuis concomitant
$\mathcal{N}(K,K)$, the result follows.
\end{proof}

A complex structure $J$ on Courant algebroid $E$ is equivalent to a decomposition of $E_{\CC}$ as a direct sum
$L\oplus\cc{L}$ of complex conjugate, maximal isotropic, involutive subbundles, namely the eigenbundles of $J$
relative to the eigenvalues $\pm\sqrt{-1}$ (see~\cite{Gualtieri}).
\begin{lemma}\label{deformation}
Let $L_{J},L_{J}^*$ be the eigenbundles of a complex structure $J$ on a Courant algebroid $E$.
Given $\Omega\in\Gamma(\wedge^{2}L_{J})$,
the subbundle $(1+\Omega\diese)L_{J}^*$ is the eigenbundle of a complex structure on $E$ if and only if
$d_{L_{J}^*}\Omega+\frac{1}{2}[\Omega,\Omega]=0$
and $\cc{\Omega}\diese\Omega\diese-\id_{L_{J}^*}$ is an invertible endomorphism of $L_{J}^*$.
\end{lemma}

A holomorphic symplectic 2-form on a Courant algebroid can be interpreted in the light of Lemma~\ref{deformation} as
a deformation of the given complex structure into a 2-dimensional sphere of complex structures (see also Proposition~\ref{S2}).
\begin{proposition}
Let $J$ be a complex structure on a Courant algebroid $E$ and let $\Omega\in\Gamma(\wedge^2 L_{J})$
be a holomorphic symplectic structure on $E$ relative to $J$.
Set $I=\Omega\diese+\cc{\Omega}\diese$ and $K=-\sqrt{-1}(\Omega\diese-\cc{\Omega}\diese)$.
Then, for any $a,b\in\mathbb{R}$, the endomorphism
$\frac{1-a^2-b^2}{1+a^2+b^2}J+\frac{2a}{1+a^2+b^2}K+\frac{2b}{1+a^2+b^2}I$
is a complex structure on $E$ and $(1+(a+b\sqrt{-1})\Omega\diese)L_{J}^{*}$ is the subbundle
associated with its eigenvalue $-\sqrt{-1}$.
\end{proposition}
Note that the map
\[ S^2\ni\left(\frac{1-a^2-b^2}{1+a^2+b^2}, \frac{2a}{1+a^2+b^2}, \frac{2b}{1+a^2+b^2}\right)
\mapsto a+b\sqrt{-1}\in\CC\cup\{\infty\}\]
is the stereographic projection.
\begin{proof}
Set $I=\Omega\diese+\cc{\Omega}\diese$ and $K=\sqrt{-1}(\cc{\Omega}\diese-\Omega\diese)$.
Then $(I,J,K)$ is a hypercomplex structure on the Courant algebroid $E$, and
$\frac{1-a^2-b^2}{1+a^2+b^2}J+\frac{2a}{1+a^2+b^2}K+\frac{2b}{1+a^2+b^2}I$ is a complex structure on $E$
by Proposition~\ref{S2}.
Since
\begin{equation}\label{conjugate}
\left(\frac{1+aI-bK}{\sqrt{1+a^{2}+b^{2}}}\right)J\left(\frac{1-aI+bK}{\sqrt{1+a^{2}+b^{2}}}\right)
=\frac{1-a^2-b^2}{1+a^2+b^2}J+\frac{2a}{1+a^2+b^2}K+\frac{2b}{1+a^2+b^2}I,
\end{equation}
and \[\frac{1-aI+bK}{\sqrt{1+a^{2}+b^{2}}}=\left(\frac{1+aI-bK}{\sqrt{1+a^{2}+b^{2}}}\right)^{-1},\]
it follows that $J$ and $\frac{1-a^2-b^2}{1+a^2+b^2}J+\frac{2a}{1+a^2+b^2}K+\frac{2b}{1+a^2+b^2}I$
are conjugate endomorphisms of the Courant algebroid $E$ and that $\frac{1+aI-bK}{\sqrt{1+a^{2}+b^{2}}}$ maps $L_{J}^{*}$
to the $-\sqrt{-1}$ eigenbundle of $\frac{1-a^2-b^2}{1+a^2+b^2}J+\frac{2a}{1+a^2+b^2}K+\frac{2b}{1+a^2+b^2}I$.

On the other hand, by Lemma~\ref{screwcap}, we have
\[(1+(a+b\sqrt{-1})\Omega\diese)\xi=\xi+a\Omega\diese\xi+b\sqrt{-1}\Omega\diese\xi=(1+aI-bK)\xi\]
for any element $\xi\in L_{J}^{*}$.

Therefore,  $(1+(a+b\sqrt{-1})\Omega\diese)L_{J}^{*}$ is the eigenbundle of
$\frac{1-a^2-b^2}{1+a^2+b^2}J+\frac{2a}{1+a^2+b^2}K+\frac{2b}{1+a^2+b^2}I$
associated with the eigenvalue $-\sqrt{-1}$.
\end{proof}

\subsection{Holomorphic symplectic structures on $T\oplus T^*$}
Consider the complex structure \[ J=\begin{pmatrix} j & 0 \\ 0 & -j^{*} \end{pmatrix} \]
on the Courant algebroid $T\oplus T^{*}$ associated with a complex manifold with complex structure $j$.
Its eigenbundles are $L_{J}=T^{1,0}\oplus(T^{0,1})^{*}$
and $L_{J}^{*}=T^{0,1}\oplus(T^{1,0})^{*}$.
Assume that $\Omega\in\Gamma(\wedge^2L_{J})$
is a holomorphic symplectic structure on $T\oplus T^{*}$ relatively to $J$.
Since \[ \wedge^{2}L_{J}=\wedge^{2}T^{1,0}\oplus
(T^{1,0}\wedge(T^{0,1})^{*})\oplus\wedge^{2}(T^{0,1})^{*} ,\]
$\Omega$ decomposes as $\Omega=\pi+\theta+\omega$
where $\pi\in\Gamma(\wedge^{2}T^{1,0})$, $\theta\in\Gamma(T^{1,0}\wedge (T^{0,1})^{*})$
and $\omega\in\Gamma(\wedge^{2} (T^{0,1})^{*})$.
Then $\Omega\diese\cc{\Omega}\diese+\cc{\Omega}\diese\Omega\diese=-1$
is equivalent to the following equations:
\begin{align*}
\pi\diese\cc{\theta}\diese+\theta\diese\cc{\pi}\diese&=0; & \pi\diese\cc{\omega}\diese+\theta\diese\cc{\theta}\diese&=-1; \\
\omega\diese\cc{\pi}\diese+\theta\diese\cc{\theta}\diese&=-1; & \omega\diese\cc{\theta}\diese+\theta\diese\cc{\omega}\diese&=0.
\end{align*}
The condition $d_{L_{J}^*}\Omega=0$ is equivalent to
\[ \cc{\partial}\pi=0; \qquad \cc{\partial}\theta=0; \qquad \cc{\partial}\omega=0, \]
as $d_{L_{J}^*}=\bar{\partial}$ in the present context. However, $[\Omega,\Omega]=0$ is equivalent to
\begin{align*} [\pi,\pi]&=0; & [\pi,\theta]&=0; \\
2[\pi,\omega]+[\theta,\theta]&=0; & [\theta,\omega]&=0 .
\end{align*}
As a consequence, $\pi$ is necessarily a holomorphic Poisson structure on $M$.

In conclusion, we have
\begin{proposition}\label{extendedsymplectic}
Given a complex manifold with complex structure $j$, if $\pi$, $\theta$, and $\omega$ are sections of
$\wedge^2 T^{1,0} $, $T^{1,0}\wedge(T^{0,1})^*$, and $\wedge^2(T^{0,1})^*$ respectively, and satisfy the relations
 \[\cc{\partial}\pi=0; \quad\cc{\partial}\theta=0; \quad\cc{\partial}\omega=0,\]
and
\begin{align*}
&\pi\diese\cc{\theta}\diese+\theta\diese\cc{\pi}\diese=0;
\qquad \pi\diese\cc{\omega}\diese+\theta\diese\cc{\theta}\diese=-1; \\
&\omega\diese\cc{\pi}\diese+\theta\diese\cc{\theta}\diese=-1;
\qquad\omega\diese\cc{\theta}\diese+\theta\diese\cc{\omega}\diese=0,
\end{align*}
then $\Omega=\pi+\theta+\omega$ is a holomorphic symplectic structure on $T\oplus T^*$
relative to $J=\big(\begin{smallmatrix} j & 0 \\ 0 & -j^{*} \end{smallmatrix}\big)$, and we have
\begin{align*}
 [\pi,\pi]&=0; & [\pi,\theta]&=0; \\
2[\pi,\omega]+[\theta,\theta]&=0; & [\theta,\omega]&=0 .
\end{align*}
In particular, $\pi$ is necessarily a holomorphic Poisson structure.
\end{proposition}


An \emph{extended Poisson structure} on a complex manifold (with complex structure $j$)
is an element $\Omega=\pi+\theta+\omega$ of $\wedge^2\big(T^{1,0}\oplus(T^{0,1})^*\big)$
satisfying $\bar{\partial}\Omega+\frac{1}{2}[\Omega,\Omega]=0$ (see~\cite{C-S-X}).
If $\Omega$ satisfies the additional algebraic condition
$\Omega\diese\cc{\Omega}\diese+\cc{\Omega}\diese\Omega\diese=-1$,
then $\Omega$ is holomorphic symplectic w.r.t.\ $J=\big(\begin{smallmatrix} j & 0 \\ 0 & -j^{*} \end{smallmatrix}\big)$
by Theorem~\ref{EqvOmega}.

This motivates the following definition.
\begin{definition}
An \emph{extended symplectic structure} on a complex manifold$(M;j)$  is a holomorphic symplectic structure
on the standard Courant algebroid $T\oplus T^{*}$ relative to the complex structure
$J=\big(\begin{smallmatrix} j & 0\\ 0 &-j^{*}\end{smallmatrix}\big)$.
\end{definition}
\begin{example}\label{exahypercomplexSymp}
If $\omega=\pi=0$, the equations in Proposition~\ref{extendedsymplectic}  become
\[ \bar{\partial}\theta=0 \quad\text{and}\quad \theta\diese\cc{\theta}\diese=-1.\]
Setting $i=\theta\diese+\cc{\theta}\diese$ and
$k=-\sqrt{-1}(\theta\diese-\cc{\theta}\diese)$, we recover a (classical) hypercomplex triple $(i,j,k)$ on the manifold
(see Example~\ref{exampleHyercomplex}).
\end{example}

\begin{example}\label{holomorphicSymplectic}
If $\theta=0$, the equations in Proposition~\ref{extendedsymplectic} become
\[ \pi=-(\cc{\omega})^{-1}, \quad \bar{\partial}\omega=0,
\quad\bar{\partial}\pi=0,
\quad [\pi,\pi]=0 \quad \text{and}\quad[\omega,\pi]=0.\]
Therefore, $\Omega=\omega+\pi$ is a holomorphic symplectic structure on $T\oplus T^{*}$
if and only if $\cc{\omega}=-\pi^{-1}$ is a holomorphic symplectic 2-form.
Thus, we recover a holomorphic symplectic manifold
(see Example~\ref{exampleHoloSymp}).
\end{example}


\subsection{Hyper-Poisson structure}
\begin{definition}[\cite{hyper-Lie-Poisson}]\label{hyperPoissondef}
Let $(i,j,k)$ be a hypercomplex triple on a manifold $M$, and let $\omega_{1},\omega_{2},\omega_{3}$ be three 2-forms
on $M$.  If $\omega_{2}+\sqrt{-1}\omega_{3}$ is a holomorphic symplectic 2-form with respect to the complex structure
$i$, $\omega_{3}+\sqrt{-1}\omega_{1}$ a holomorphic symplectic 2-form with respect to the complex structure $j$,
and $\omega_{1}+\sqrt{-1}\omega_{2}$ a holomorphic symplectic 2-form with respect to the complex structure $k$,
then $(\omega_{1},\omega_{2},\omega_{3})$ is a \emph{hyper-symplectic structure} on $M$ with respect to $(i,j,k)$.
\end{definition}

\begin{remark}
Definition \ref{hyperPoissondef} is invariant under cyclic permutations of $i,j,k$ and $\omega_1,\omega_2,\omega_3$.
\end{remark}

The following definition is a natural generalization of hyper-symplectic structures in the Poisson context.
\begin{definition}
Let $(i,j,k)$ be a hypercomplex triple on a manifold $M$, and let $\pi_{1},\pi_{2},$ and $\pi_{3}$ be three bivector fields on $M$.
If $\pi_{2}-\sqrt{-1}\pi_{3}$ is a holomorphic Poisson tensor with respect to the complex structure $i$,
$\pi_{3}-\sqrt{-1}\pi_{1}$ a holomorphic Poisson tensor with respect to the complex structure $j$,
and $\pi_{1}-\sqrt{-1}\pi_{2}$ a holomorphic Poisson tensor with respect to the complex structure $k$,
then $(\pi_{1},\pi_{2},\pi_{3})$ is  a \emph{hyper-Poisson structure} on $M$ with respect to $(i,j,k)$.
\end{definition}

\begin{theorem}\label{prophyperpoisson}
 The following assertions are equivalent to each other:
\begin{enumerate}
\item The triple $(i,j,k)$ is a hypercomplex structure on $M$ and $(\pi_{1},\pi_{2},\pi_{3})$
is a hyper-Poisson structure with respect to $(i,j,k)$ satisfying $\pi_{2}\diese=-i\pi_{3}\diese=-\pi_{3}\diese i^{*}$.
\item The triple $I=\big(\begin{smallmatrix} i & \pi_{3} \\ 0 & -i^{*} \end{smallmatrix}\big)$,
$J=\big(\begin{smallmatrix} j & 0 \\ 0 & -j^{*} \end{smallmatrix}\big)$,
$K=\big(\begin{smallmatrix} k & -\pi_{1} \\ 0 & -k^{*} \end{smallmatrix}\big)$
is a hypercomplex structure on the standard Courant algebroid $T_M\oplus T^{*}_M$.
\item
Having defined $\theta\in\Gamma(T^{1,0}\wedge(T^{0,1})^{*})$ (relatively to $j$) by $\theta\diese=\frac{1}{2}(i+\sqrt{-1}k)$
and $\pi\in\Gamma(\wedge^{2}T^{1,0})$ (relatively to $j$) by $\pi=\frac{1}{2}(\pi_{3}-\sqrt{-1}\pi_{1})$,
their sum $\Omega=\theta+\pi$ is a holomorphic symplectic structure on the standard Courant algebroid
$T_M\oplus T^{*}_M$ relative to the complex structure $J= \big(\begin{smallmatrix}j & 0 \\ 0 & -j^{*} \end{smallmatrix}\big)$.
\end{enumerate}
\end{theorem}


In order to prove Theorem~\ref{prophyperpoisson}, we will make use of the following lemma or more precisely a corollary of it.
\begin{lemma}[\cite{holomorphic-Lie-algebrods}]\label{lemma-holomorphicPoisson}
Let $M$ be a complex manifold (with complex structure $j$).
If $\pi_{\lambda}$ and $\pi_{\mu}$ are two real bivector fields on $M$
such that $\pi_{\lambda}+\sqrt{-1}\pi_{\mu}$ is a holomorphic Poisson bivector field,
then $\pi_{\mu}\diese=-j\pi_{\lambda}\diese=-\pi_{\lambda}\diese j^{*}$ and
\[\llbracket\pi_{\lambda},\pi_{\lambda}\rrbracket=\llbracket\pi_{\mu},
\pi_{\mu}\rrbracket=\llbracket\pi_{\lambda},\pi_{\mu}\rrbracket=0. \]
\end{lemma}
\begin{corollary}\label{corollary-hyperPoisson}
Given a hyper-Poisson structure $(\pi_{1},\pi_{2},\pi_{3})$ on a manifold $M$ with respect to a hypercomplex triple $(i,j,k)$,
we have
\begin{enumerate}
\item $\llbracket \pi_{\alpha},\pi_{\beta}\rrbracket=0$, for any $\alpha,\beta\in\{1,2,3\}$;
\item $\pi_{3}\diese=i\pi_{2}\diese=\pi_{2}\diese i^{*}$,
$\pi_{1}\diese=j\pi_{3}\diese=\pi_{3}\diese j^{*}$,
$\pi_{2}\diese=k\pi_{1}\diese=\pi_{1}\diese  k^{*}$;
\item $i\pi_{1}\diese=-\pi_{1}\diese i^{*}=j\pi_{2}\diese
=-\pi_{2}\diese j^{*}=k\pi_{3}\diese=-\pi_{3}\diese k^{*}$.
\end{enumerate}
Therefore, if one among $\pi_{1}$, $\pi_{2}$, and $\pi_{3}$ is invertible, then so are the other two.
In this case, $(\pi_{1}^{-1},\pi_{2}^{-1},\pi_{3}^{-1})$ is a hyper-symplectic structure on $M$
and $(i\pi_{1}\diese)^{-1}$ defines a pseudo-metric $g$ on $M$.
In particular, if $g$ is positive definite, we obtain a hyper-K\"{a}hler structure on $M$.
\end{corollary}

\begin{proof}[Proof of Theorem~\ref{prophyperpoisson}]
\quad $(1)\Rightarrow(2)$\quad
Assume that $(\pi_1,\pi_2,\pi_3)$ is a hyper-Poisson  structure on $M$ with respect to a hypercomplex triple $(i,j,k)$.
As $\pi_{3}+\sqrt{-1}\pi_{2}=\sqrt{-1}(\pi_{2}-\sqrt{-1}\pi_{3})$ is a holomorphic
Poisson structure with respect to the complex structure $i$, by Lemma \ref{holomorphicPoissonGC},
$I=\big(\begin{smallmatrix} i & \pi_{3} \\ 0 & -i^{*} \end{smallmatrix}\big)$ is a complex structure
on the standard Courant algebroid $TM\oplus T^{*}M$
and $\pi_{2}\diese=-i\pi_{3}\diese=-\pi_{3}\diese i^{*}$.
Similarly, $K=\big(\begin{smallmatrix} k & -\pi_{1} \\ 0 & -k^{*} \end{smallmatrix}\big)$
is a complex structure on $TM\oplus T^{*}M$.
It follows from Corollary~\ref{corollary-hyperPoisson}
that $(I,J,K)$ is an almost hypercomplex structure on $TM\oplus T^{*}M$, and from
Theorem~\ref{eqv} that $(I,J,K)$ is a hypercomplex structure on $TM\oplus T^{*}M$.

$(2)\Leftrightarrow(3)$\quad
The equivalence of (2) and (3) is a direct consequence of Theorem~\ref{hyper-holosym} and Proposition~\ref{extendedsymplectic}.

$(2)\Rightarrow(1)$\quad
Given a hypercomplex structure $I=\big(\begin{smallmatrix} i & \pi_{3} \\ 0 & -i^{*} \end{smallmatrix}\big),
J=\big(\begin{smallmatrix} j & 0 \\ 0 & -j^{*} \end{smallmatrix}\big),
K=\big(\begin{smallmatrix} k & -\pi_{1} \\ 0 & -k^{*} \end{smallmatrix}\big)$
on the standard Courant algebroid $E=TM\oplus T^{*}M$,
let $\theta\in\Gamma(T^{1,0}\wedge (T^{0,1})^{*})$ and $\pi\in\Gamma(\wedge^{2}T^{1,0})$ be
defined by $\theta\diese=\frac{1}{2}(i+\sqrt{-1}k)$ and $\pi=\frac{1}{2}(\pi_{3}-\sqrt{-1}\pi_{1})$.
Then $\Omega=\theta+\pi$ is a holomorphic symplectic form on $E$
with respect to $J=\big(\begin{smallmatrix} j & 0 \\ 0 & -j^{*} \end{smallmatrix}\big)$, and, by Proposition~\ref{extendedsymplectic}, we have
\[\bar{\partial}\theta=\bar{\partial}\pi=0,
\quad [\pi,\pi]=0, \quad [\pi,\theta]=0, \quad [\theta,\theta]=0, \quad
\theta\diese\cc{\theta}\diese=-1, \quad
\pi\diese\cc{\theta}\diese+\theta\diese\cc{\pi}\diese=0.\]
Thus, $(i,j,k)$ is a hypercomplex structure on $M$ by our discussion in Example~\ref{exahypercomplexSymp}, and
$\pi_{3}-\sqrt{-1}\pi_{1}=2\pi$ is a holomorphic Poisson structure with respect to $j$.
As $I=\big(\begin{smallmatrix} i & \pi_{3} \\ 0 & -i^{*} \end{smallmatrix}\big)$ is a complex structure on $E=TM\oplus T^{*}M$, by Lemma~\ref{holomorphicPoissonGC},  $\pi_{2}-\sqrt{-1}\pi_{3}=-\sqrt{-1}(\pi_{3}+\sqrt{-1}\pi_{2})$
is a holomorphic Poisson structure with respect to the complex structure $i$,
where $\pi_{2}$ is defined by $\pi_{2}\diese=-i\pi_{3}\diese=-\pi_{3}i^{*}$.
And, $\pi_{1}-\sqrt{-1}\pi_{2}$ is a holomorphic Poisson structure with respect to $k$, since
$K=\big(\begin{smallmatrix} k & -\pi_{1} \\ 0 &k^{*}\end{smallmatrix}\big)$
is a complex structure on $E=TM\oplus T^{*}M$ and
$\pi_{2}\diese=-i\pi_{3}\diese=ij\pi_{1}\diese=k\pi_{1}\diese=\pi_{1}\diese k^{*}$.
\end{proof}

\begin{corollary}
If $(\pi_{1},\pi_{2},\pi_{3})$ is a hyper-Poisson structure on a smooth manifold $M$
with respect to a hypercomplex structure $(i,j,k)$,
then $\pi_{1}$, $\pi_{2}$, and $\pi_{3}$ have the same symplectic leaves and each leaf is a hyper-symplectic manifold.
\end{corollary}
\begin{proof}
As $\pi_{2}\diese=-\pi_{3}\diese i^{*}$ according to Theorem~\ref{prophyperpoisson},  it follows that
$\pi_{2}\diese(T^{*}M)=\pi_{3}\diese(T^{*}M)$. Therefore, $\pi_{2}$ and $\pi_{3}$ have the same symplectic foliations.
Moreover, since $\pi_{2}\diese=-i\pi_{3}\diese$, such symplectic foliations are stable under $i$. Hence, each symplectic leaf
of $\pi_{2}$ (and $\pi_{3}$) is a complex submanifold with respect to $i$.
As the rule of $\pi_{1}, \pi_{2}, \pi_{3}$ is symmetric, the conclusion follows.
\end{proof}

\section{Hypercomplex connection and Lagrangian Lie subalgebroid}

\subsection{A basic study of hypercomplex connection}
\begin{lemma}\label{nabla-part}
Let $(I,J,K)$ be a hypercomplex structure on a Courant algebroid $(E,\rho,\langle,\rangle,\circ)$
and let $\Omega\in\wedge^{2}L_{J}$ denote the corresponding holomorphic symplectic form relative to $J$
defined by $\Omega\diese=\frac{I+iK}{2}$.
For any $X,Y,Z\in\Gamma(L_{J})$ and $\xi,\eta,\zeta\in\Gamma(L_{J}^{*})$, we have
\begin{enumerate}
\item $\nabla_{X}\xi=\imath_{X}d_{L_{J}}(\xi)\in\Gamma(L_{J}^{*})$ and
$\nabla_{\xi}X=\imath_{\xi}d_{L_{J}^*}(X)\in\Gamma(L_{J})$,
\item $\nabla_{X}Y=-\Omega\diese(\imath_{X}\mathcal{L}_{Y}\cc{\Omega})\in\Gamma(L_{J})$ and
$\nabla_{\xi}\eta=-\cc{\Omega}\diese(\imath_{\xi}\mathcal{L}_{\eta}\Omega)\in\Gamma(L_{J}^{*})$,
\item $R^{\nabla}(X,Y)Z=0$ and $R^{\nabla}(\xi,\eta)\zeta=0$,
\end{enumerate}
where $\nabla$ stands for the hypercomplex connection defined in Equation~\eqref{connection}
and $R^{\nabla}$ for its curvature.
\end{lemma}

\begin{proof}
\quad\textbf{(1)}\quad For all $X, Y\in\Gamma(L_{J}), \xi\in\Gamma(L_{J}^*)$, by Equation \eqref{connection}, we have
\begin{equation*}
  \langle\nabla_{X}\xi, Y\rangle
=\langle-\frac{1+iJ}{2}(\xi\circ X)+\frac{K}{2}(i+J)(\xi\circ IX), Y\rangle
=-\langle\xi\circ X, Y\rangle,
\end{equation*}
as $I$ swaps $L_{J}$ and $L_{J}^*$,  and as $L_{J}^*$ is involutive. Thus,
\begin{align*}
 &\langle\nabla_{X}\xi, Y\rangle=\langle-\xi\circ X, Y\rangle
=\langle-2\mathcal{D}\langle\xi, X\rangle+X\circ\xi, Y\rangle\\
=& -\rho(Y)\langle\xi, X\rangle+\rho(X)\langle\xi, Y\rangle-\langle\xi, X\circ Y\rangle
=d_{L_{J}}(\xi)(X, Y),
\end{align*}
Now $\nabla_{X}\xi=\imath_{X}d_{L_{J}}(\xi)$, as $\nabla J=0$ implies that $\nabla_{X}\xi\in L_{J}^{*}$.

Similarly, we have $\nabla_{\xi}X=\imath_{\xi}d_{L_{J}^*}(X)$.

\textbf{(2)}\quad First, $\nabla J=0$ implies that $\nabla_{\xi}\eta\in\Gamma(L_{J}^*)$, $\forall\xi,\eta\in\Gamma(L_{J}^*)$. By Equation \eqref{connection}, we have
\begin{equation*}
\nabla_{\xi}\eta
=\frac{-I+iK}{2}(\eta\circ I\xi)-\eta\circ\xi=-\cc{\Omega}\diese(\eta\circ\Omega\diese\xi)-\eta\circ\xi.
\end{equation*}
As $\Omega^{\sharp}:L^{*}_{J}\rightarrow L_{J}$ is the inverse of
$-\cc{\Omega}^{\sharp}:L_{J}\rightarrow L_{J}^{*}$, for all $\xi,\eta,\zeta\in\Gamma(L_{J}^*)$, we have
\begin{equation*}
\Omega(\nabla_{\xi}\eta,\zeta)
=\langle\eta\circ\Omega\diese\xi,\zeta\rangle-\Omega(\eta\circ \xi,\zeta)
=(\mathcal{L}_{\eta}\Omega)(\xi,\zeta).
\end{equation*}
Therefore,
\begin{equation}\label{nabla-xieta}
(\nabla_{\xi}\eta)\lrcorner\Omega=\xi\lrcorner(\mathcal{L}_{\eta}\Omega),
\end{equation}
i.e.\ $\Omega\diese(\nabla_{\xi}\eta)=\imath_{\xi}(\mathcal{L}_{\eta}\Omega)$. From Remark \ref{Oginvert}, $\nabla_{\xi}\eta=-\cc{\Omega}\diese(\imath_{\xi}\mathcal{L}_{\eta}\Omega)$ follows.

Similarly, we have $\nabla_{X}Y=-\Omega\diese(\imath_{X}\mathcal{L}_{Y}\cc\Omega)\in\Gamma(L_{J})$
for all $X, Y\in\Gamma(L_{J})$.


\textbf{(3)}\quad For all $\xi,\eta,\zeta\in\Gamma(L_{J}^*)$, by Equation~\eqref{nabla-xieta}, we have
\begin{align*}
R^{\nabla}(\xi,\eta)\zeta\lrcorner\Omega
=&(\nabla_{\xi}\nabla_{\eta}\zeta-\nabla_{\eta}\nabla_{\xi}\zeta-\nabla_{\llbracket\xi,\eta\rrbracket}\zeta)\lrcorner\Omega\\
=&\xi\lrcorner d_{L_{J}^*}(\nabla_{\eta}\zeta\lrcorner\Omega)-\eta\lrcorner d_{L_{J}^*}(\nabla_{\xi}\zeta\lrcorner\Omega)
-\llbracket\xi,\eta\rrbracket\lrcorner d_{L_{J}^*}(\zeta\lrcorner\Omega)\\
=&\xi\lrcorner d_{L_{J}^*}(\eta\lrcorner d_{L_{J}^*}(\zeta\lrcorner\Omega))
-\eta\lrcorner d_{L_{J}^*}(\xi\lrcorner d_{L_{J}^*}(\zeta\lrcorner\Omega))
-\llbracket\xi,\eta\rrbracket\lrcorner d_{L_{J}^*}(\zeta\lrcorner\Omega)\\
=&\xi\lrcorner\mathcal{L}_{\eta} d_{L_{J}^*}(\zeta\lrcorner\Omega)
-\eta\lrcorner\mathcal{L}_{\xi} d_{L_{J}^*}(\zeta\lrcorner\Omega)
-\llbracket\xi,\eta\rrbracket\lrcorner d_{L_{J}^*}(\zeta\lrcorner\Omega)\\
=&(\xi\lrcorner\mathcal{L}_{\eta}-\eta\lrcorner\mathcal{L}_{\xi}
-\llbracket\xi,\eta\rrbracket\lrcorner)d_{L_{J}^*}(\zeta\lrcorner\Omega)\\
=&0 .
\end{align*}
Thus,  $R^{\nabla}(\xi,\eta)\zeta=0$, $\forall \xi,\eta,\zeta\in\Gamma(L_{J}^*)$.

Similarly, $R^{\nabla}(X,Y)Z=0$, $\forall X,Y,Z\in\Gamma(L_{J})$.
\end{proof}

\begin{proposition}\label{connection-eqv}
Let $(I,J,K)$ be a hypercomplex structure on a Courant algebroid $(E,\rho,\langle,\rangle,\circ)$,
let $\nabla$ be the associated hypercomplex connection given in Equation~\eqref{connection},
and let $\Omega\in\wedge^{2}L_{J}$ be the symplectic form, holomorphic with respect to $J$,
defined by $\Omega\diese=\frac{I+iK}{2}$.
The following assertions are equivalent to each other for any $V\in\Gamma(E)$:
\begin{enumerate}
\item $\nabla V=0$;
\item $d_{L_{I}}(V+iIV)=0$, $d_{L_{J}}(V+iJV)=0$, and $d_{L_{K}}(V+iKV)=0$;
\item $d_{L_{J}}(V+iJV)=0$ and $\mathcal{L}_{V+iJV}\Omega=0$.
\end{enumerate}
(Here we consider $V+iIV$ as a section of $\Gamma(\cc{L_{I}})$,
$V+iJV$ as a section of $\Gamma(\cc{L_{J}})$,
and $V+iKV$ as a section of $\Gamma(\cc{L_{K}})$.)
\end{proposition}

\begin{proof} Set $\xi=V+iJV$, then $\xi\in\Gamma(L_J^*)$.

\quad (1)$\Rightarrow$(2) \quad
Given $\nabla V=0$, we have $\nabla\xi=0$, as $\nabla J=0$. By Lemma~\ref{nabla-part}, $\imath_{X}d_{L_{J}}(\xi)=\nabla_{X}\xi=0$, for all $X\in\Gamma(L_{J})$. As a consequence, $d_{L_{J}}(V+iJV)=d_{L_{J}}(\xi)=0.$
Similarly, we have $d_{L_{I}}(V+iIV)=0$ and $d_{L_{K}}(V+iKV)=0$.

\quad (2)$\Rightarrow$(1) \quad
For any $ U\in\Gamma(E)$, as $U-iJU$ is a section of $L_{J}$, by Lemma \ref{nabla-part}, we have 
\[\nabla_{U-iJU}(V+iJV)=\imath_{(U-iJU)}d_{L_{J}}(\xi)=0.\]  Since
\[\nabla_{U-iJU}(V+iJV)=(\nabla_{U}V+\nabla_{JU}JV)+i(\nabla_{U}JV-\nabla_{JU}V),\]
and since $\nabla$ defined by Equation~\eqref{connection} is a real connection, we have $\nabla_{U}JV=\nabla_{JU}V$.
Hence, by $\nabla J=0$, we have
\begin{equation}\label{nabla_J}
\nabla_{JU}V=\nabla_{U}JV=J\nabla_{U}V.
\end{equation}
Similarly,
\begin{equation}\label{nabla_IK}
\nabla_{IU}V=I\nabla_{U}V,~~\nabla_{KU}V=K\nabla_{U}V.
\end{equation}
From Equations \eqref{nabla_J} and \eqref{nabla_IK} and identity $\nabla I=\nabla J=\nabla K=0$, we have
\begin{equation}
K\nabla_{U}V=\nabla_{U}KV=\nabla_{U}IJV=I\nabla_{U}(JV)=\nabla_{IU}JV=J\nabla_{IU}V=JI\nabla_{U}V=-K\nabla_{U}V.
\end{equation}
Therefore, $\nabla_{U}V=0$ for all $U\in\Gamma(E)$, i.e.\ $\nabla V=0$.

\quad (1)$\Rightarrow$(3) \quad
As $\nabla V=0$ and $\nabla J=0$, we have $\nabla\xi=0$.
It follows from Lemma~\ref{nabla-part} that $\nabla_{X}\xi=\imath_{X}d_{L_{J}}(\xi)$  and $\nabla_{\eta}\xi=-\cc{\Omega}\diese(\imath_{\eta}\mathcal{L}_{\xi}\Omega)$ for all $X\in\Gamma(L_{J})$ and $\eta\in\Gamma(L_{J}^{*})$.
Thus, $d_{L_{J}}(\xi)=0$ and $\mathcal{L}_{\xi}\Omega=0$.

\quad (3)$\Rightarrow$(1) \quad
Since $d_{L_{J}}(\xi)=0$ and $\mathcal{L}_{\xi}\Omega=0$, by Lemma~\ref{nabla-part}, we have $\nabla\xi=0$.
As $\nabla$ is a real connection, we have $\nabla V=0$ from $\nabla(V+iJV)=\nabla\xi=0$.
\end{proof}

\begin{remark}
In Proposition~\ref{connection-eqv}, (2) can be understood as the requirement that the vector field $V$ be the common real component of three complex vector fields, which are holomorphic
with respect to the complex structures $I$, $J$, and $K$ respectively, whereas (3) can be understood as the requirement that $V$ be the real component of a holomorphic symplectic vector field.
\end{remark}

\subsection{The induced connections on Lagrangian Lie subalgebroids}
\begin{definition}
Let $(E,\rho,\langle,\rangle,\circ)$ be a Courant algebroid endowed with a complex structure $J$.
Assume that $\Omega\in\Gamma(\wedge^{2}L_{J})$ is a holomorphic symplectic structure on $E$
relative to the complex structure $J$. A Lie subalgebroid of $L_J^*$ is said to be \emph{Lagrangian}
if it is maximal isotropic with respect to $\Omega$.
\end{definition}

\begin{theorem}\label{connecontr}
Let $(E,\rho,\langle,\rangle,\circ)$ be a Courant algebroid endowed with a complex structure $J$.
Assume that $\Omega\in\Gamma(\wedge^2 L_{J})$ is holomorphic symplectic with respect to $J$, and set
$I=\Omega\diese+\cc{\Omega}\diese$ and $K=-i(\Omega\diese-\cc{\Omega}\diese)$.
Given a Lie subalgebroid $L$ of $L_J^*$ Lagrangian with respect to $\Omega$,
the hypercomplex connection $\nabla$ defined by Equation~\eqref{connection}
induces a torsion-free flat $L$-connection $\nabla|_{L}$ on $L$.
\end{theorem}

\begin{proof}
For all $\xi,\eta,\zeta\in\Gamma(L)$, as $L$ is involutive and isotropic with respect to $\Omega$, we have $\Omega(\xi,\zeta)=\Omega(\xi,\eta\circ\zeta)=\Omega(\eta\circ\xi,\zeta)=0$.
By Equation~\eqref{nabla-xieta}, we have
\begin{equation*}
\Omega(\nabla_{\xi}\eta,\zeta)=(\mathcal{L}_{\eta}\Omega)(\xi,\zeta)
=\rho(\eta)\Omega(\xi,\zeta)-\Omega(\xi,\eta\circ\zeta)-\Omega(\eta\circ\xi,\zeta)=0.
\end{equation*}
As $L$ is a Lagrangian subalgebroid with respect to $\Omega$, we get $\nabla_{\xi}\eta\in\Gamma(L)$ for all $\xi,\eta\in\Gamma(L)$.

On the other hand, for all $\xi,\eta\in\Gamma(L)$, we have $\langle\xi,\eta\rangle=0$,
$\langle I\xi,\eta\rangle=\langle\Omega\diese\xi,\eta\rangle=\Omega(\xi,\eta)=0$,
$\langle J\xi,\eta\rangle=-i\langle\xi,\eta\rangle=0$,
and $\langle K\xi,\eta\rangle=-i\langle\Omega\diese\xi,\eta\rangle=-i\Omega(\xi,\eta)=0$.
It follows from Equations \eqref{deltaf}, \eqref{fnabla}, \eqref{nablaf}, and \eqref{torsion}
that $\nabla|_{L}$ is a torsion-free $L$-connection on $L$ and from Lemma~\ref{nabla-part}
that $\nabla|_{L}$ is flat.
\end{proof}

The following lemma, which demonstrates the relation between Dirac structures and Lagrangian Lie subalgebroids, is easy to verify.
\begin{lemma}\label{DiracLagragian}
Let $(I, J, K)$ be a hypercomplex structure on a Courant algebroid $(E,\rho,\langle,\rangle,\circ)$, and
let $\Omega\in\Gamma(\wedge^{2}L_{J})$ be the associated holomorphic symplectic form relative to $J$.
If $D$ is a Dirac subbundle of $E$ stable under $I$, $J$, and $K$, then $L=\frac{1+iJ}{2}D$
is a subalgebroid of $L_J^*$ Lagrangian with respect to $\Omega$.
\end{lemma}

As consequences of Theorem~\ref{connecontr} and Lemma~\ref{DiracLagragian},
we consider the following two special cases.

Let $M$ be a complex manifold (with complex structure $j$),
and let $\omega$ be a holomorphic symplectic 2-form on $M$.
Set $J=\big(\begin{smallmatrix} -j & 0 \\ 0 & j^{*} \end{smallmatrix}\big)$ (contrast Example~\ref{holomorphicSymplectic})
and $\Omega=\omega+\pi$ with $\pi=-(\cc{\omega})^{-1}$.
As shown in Example~\ref{ExampleLagrangianFoliation}, if $S$ is a complex Lagrangian foliation,
then $D=T_{S}\oplus T_{S}^{\perp}$ is a Dirac structure stable under $I$, $J$, and $K$.
Hence, the Lie subalgebroid $L=\frac{1+\sqrt{-1}J}{2}D=T_{S}^{1,0}\oplus(T_{S}^{\perp})^{0,1}$
of $L_{J}^*=T^{1,0}\oplus (T^{0,1})^*$ is Lagrangian with respect to $\Omega=\omega+\pi$
according to Lemma~\ref{DiracLagragian}. By Theorem~\ref{connecontr},
the restriction of the hypercomplex connection $\nabla$  (defined by Equation~\eqref{connection})
to $L$ is a flat torsion-free connection that satisfies
\[ \nabla_{X}Y=\omega^{-1}(\imath_{X}\partial(\omega(Y)))
\in\Gamma(T^{1,0}_{S}),\quad\forall X,Y\in\Gamma(T^{1,0}_{S}) .\]
Thus, we recover the flat torsion-free connection obtained by Behrend \& Fantechi in~\cite{Kai}.

\begin{corollary}\label{LagrangianFoliation}
Let $S$ be a complex Lagrangian foliation of a holomorphic symplectic manifold $(M;j,\omega)$.
Then
\[ \nabla_{X}Y=\omega^{-1}(\imath_{X}\partial(\omega(Y)))
\in\Gamma(T^{1,0}_{S}),\quad\forall X,Y\in\Gamma(T^{1,0}_{S}) \]
defines a torsion-free flat $T^{1,0}_{S}$-connection on $T^{1,0}_{S}$.
\end{corollary}

Now we consider another special case. As in Example~\ref{exahypercomplexSymp},
consider a complex manifold $M$ (with complex structure $j$),
a holomorphic symplectic structure $\Omega=\theta\in\OO^{0,1}(T^{1,0})$
on $T\oplus T^*$ with respect to $J=\big(\begin{smallmatrix} j & 0 \\ 0 & -j^{*} \end{smallmatrix}\big)$,
and the associated hypercomplex triple $(i,j,k)$ on $X$.
As shown in Example~\ref{ExampleHypercomplexFoliation}, if $S$ is a hypercomplex foliation,
then $D=T_{S}\oplus T_{S}^{\perp}$ is a Dirac structure stable under $I$, $J$, and $K$.
Hence, the Lie subalgebroid $L=\frac{1+\sqrt{-1}J}{2}D=T_{S}^{0,1}\oplus (T_{S}^{\perp})^{1,0}$
of $L_{J}^*=T^{0,1}\oplus (T^{1,0})^*$ is Lagrangian with respect to $\Omega=\theta$ according to Lemma~\ref{DiracLagragian}. By Theorem~\ref{connecontr}, the restriction of the hypercomplex connection $\nabla$ (defined by Equation~\eqref{connection}) to $L$ is a flat torsion-free connection that satisfies
\begin{equation}\label{theta-conjugation}
\nabla_{\cc{X}}\cc{Y}=-\bar{\theta}(\imath_{\cc{X}}\bar{\partial}(\theta(\cc{Y})))
\in\Gamma(T^{0,1}_{S}),\quad\forall \cc{X},\cc{Y}\in\Gamma(T^{0,1}_{S}) .
\end{equation}
If we consider the conjugation of  Equation \eqref{theta-conjugation}, then we get the following corollary.
\begin{corollary}\label{hypercomplexfoliation}
Let $S$ be a hypercomplex foliation on a hypercomplex manifold $(M;i,j,k)$,
and let  $\theta\in\OO^{0,1}(T^{1,0})$ (relatively to $j$)  be defined by $\theta\diese=\frac{1}{2}(i+\sqrt{-1}k)$. Then
 \[ \nabla_{X}Y=-\theta(\imath_{X}\partial(\bar{\theta}(Y)))
\in\Gamma(T^{1,0}_{S}),\quad\forall X,Y\in\Gamma(T^{1,0}_{S}) \]
defines a torsion-free flat $T^{1,0}_{S}$-connection on $T^{1,0}_{S}$.
\end{corollary}

\bibliographystyle{amsplain}
\bibliography{hypercomplex-9}

\end{document}